\documentclass[10pt,reqno]{amsart}
\usepackage{graphicx}
\usepackage{tabularx}
\usepackage{amsfonts}
\usepackage{amssymb}
\usepackage{amsmath}
\usepackage{amsxtra}
\usepackage{latexsym}
\usepackage{epstopdf}
\usepackage{mathrsfs}
\usepackage{esint}
\usepackage{mathrsfs,galois,booktabs,ctable,threeparttable,tabularx,algorithm,algorithmic}
\usepackage{enumitem}
\usepackage{multirow,booktabs}
\usepackage{float}
\usepackage{graphicx}
\usepackage{subfigure} 
\usepackage{setspace}
\usepackage{hyperref}
\graphicspath{ {./figures/} }
\newtheorem{theorem}{Theorem}[section]
\newtheorem{lemma}[theorem]{Lemma}

\newtheorem{proposition}[theorem]{Proposition}

\theoremstyle{definition}

\newtheorem{example}{Example}[section]

\theoremstyle{remark}
\newtheorem{remark}{Remark}[section]

\numberwithin{equation}{section}

\begin{document}

\title[]
{}

\title[SVRG for linear ill-posed problems]
{Stochastic variance reduced gradient method for linear ill-posed inverse problems}

\author{Qinian Jin}
\address{Mathematical Sciences Institute, Australian National
University, Canberra, ACT 2601, Australia}
\email{qinian.jin@anu.edu.au} \curraddr{}

\author{Liuhong Chen}
\address{School of Mathematics and Statistics, Huazhong University of Science and Technology,
Wuhan 430074, China}
\email{liuhong$_{-}$c@hust.edu.cn}




\begin{abstract}
In this paper we apply the stochastic variance reduced gradient (SVRG) method, which is a popular variance reduction method in optimization for accelerating the stochastic gradient method, to solve large scale linear ill-posed systems in Hilbert spaces. Under {\it a priori} choices of stopping indices, we derive a convergence rate result when the sought solution satisfies a benchmark source condition and establish a convergence result without using any source condition. To terminate the method in an {\it a posteriori} manner, we consider the discrepancy principle and show that it terminates the method in finite many iteration steps almost surely. Various numerical results are reported to test the performance of the method. 
\end{abstract}


\def\d{\delta}
\def\E{\mathbb{E}}
\def\P{\mathbb{P}}
\def\X{\mathcal{X}}
\def\A{\mathcal{A}}
\def\Y{\mathcal{Y}}
\def\l{\langle}
\def\r{\rangle}
\def\by{\overline{y}^{(n)}}
\def\bY{\overline{y}_n}
\def\la{\lambda}
\def\EE{{\mathbb E}}
\def\RR{{\mathbb R}}
\def\a{\alpha}
\def\l{\langle}
\def\r{\rangle}
\def\p{\partial}
\def\ep{\varepsilon}
\def\O{\mathcal O}
\def\M{\mathcal M}
\def\F{{\mathcal F}}
\def\PP{{\mathbb P}}
\def\prox{\mbox{prox}}

%
%


\maketitle

\section{\bf Introduction}

Consider ill-posed inverse problems governed by the linear system  
\begin{align}\label{SVRG.1}
A_i x = y_i, \quad i = 1, \cdots, N,
\end{align}
where, for each $i = 1, \cdots, N$, $A_i : X \to Y_i$ is a bounded linear operator from 
a fixed Hilbert space $X$ to a Hilbert space $Y_i$. Here, ill-posedness means the solution of (\ref{SVRG.1}) does not depend continuously on the data. Such problems
arise in a broad range of applications including various tomography imaging and inverse 
problems with discrete data (\cite{BDP1985,N2001}). Let $Y:= Y_1 \times \cdots \times Y_N$ 
be the product space of $Y_1, \cdots, Y_N$ with the natural inner product inherited 
from those of $Y_i$. Let $A: X \to Y$ be defined by 
$$
A x := (A_1 x, \cdots, A_N x), \quad \forall x \in X. 
$$
Then $A$ is a bounded linear operator and (\ref{SVRG.1}) can be written as 
$
A x = y
$
with $y:= (y_1, \cdots, y_N)\in Y$. Note that the adjoint $A^*: Y \to X$ of $A$ is given by 
$$
A^* z = \sum_{i=1}^N A_i^* z_i, \quad \forall z = (z_1, \cdots, z_N) \in Y,
$$
where $A_i^*: Y_i \to X$ denotes the adjoint of $A_i$ for each $i$. 
In what follows we always assume that (\ref{SVRG.1}) has a solution, i.e. $y \in \mbox{Ran}(A)$, 
the range of $A$. By taking an initial guess $x_0\in X$. we aim at finding a solution $x^\dag$ of 
(\ref{SVRG.1}) such that 
$$
\|x^\dag - x_0\| = \min\{\|x - x_0\|: A_i x = y_i, \, i = 1, \cdots, N\}.
$$
It is easy to see that this solution $x^\dag$ exists and is unique; we will call $x^\dag$ the 
$x_0$-minimal norm solution of (\ref{SVRG.1}). It is known that $x^\dag$ is the $x_0$-minimal norm 
solution of (\ref{SVRG.1}) if and only if $A x^\dag = y$ and $x^\dag - x_0 \in \mbox{Null}(A)^\perp$, 
where $\mbox{Null}(A)$ denotes the null space of $A$. i.e. $\mbox{Null}(A):=\{x\in X: A x =0\}$. 

In practical applications, data are usually acquired by measurements. Therefore, 
instead of the exact data $y = (y_1, \cdots, y_N)$, we have only noisy data 
$y^\d:= (y_1^\d, \cdots, y_N^\d)$ satisfying 
\begin{align}\label{noise}
\|y^\d - y\| := \left(\sum_{i=1}^N \|y_i^\d - y_i\|^2\right)^{1/2} \le \d, 
\end{align}
where $\d>0$ denotes the noise level. It is therefore important to develop algorithms to compute 
$x^\dag$ approximately using the noisy data $y^\d$. Many regularization methods have been proposed for this purpose in the literature, see \cite{EHN1996}. The most prominent iterative regularization method is the Landweber method
\begin{align}\label{LM}
    x_{n+1}^\d = x_n^\d - \gamma A^*(A x_n^\d - y^\d)
\end{align}
where $x_0^\d := x_0\in X$ is an initial guess and $\gamma>0$ is a step-size. It is known that if $0<\gamma \le 2/\|A\|^2$ then Landweber method, terminated by the discrepancy principle, is an order optimal regularization method. Because of its simple implementation and low complexity per iteration, Landweber method is popular for solving ill-posed inverse problems. 

Note that the implementation of the Landweber method (\ref{LM}) at each iteration step requires to calculate 
$$
A^*(A x_n^\d - y^\d) = \sum_{i=1}^N A_i^*(A_i x_n^\d - y_i^\d).
$$
In case $N$ is huge, this requires a huge amount of computational time because of the calculation of $A_i^*(A_i x_n^\d - y_i^\d)$ for all $i$. To resolve this issue, the stochastic gradient descent method, which is popular for solving large scale optimization problems, has been utilized to solve ill-posed inverse problems of the form (\ref{SVRG.1}) in recent years and the method takes the form 
\begin{align}\label{SGD}
    x_{n+1}^\d = x_n^\d - \gamma_n A_{i_n}^*(A_{i_n} x_n^\d - y_{i_n}^\d),
\end{align}
where $i_n \in \{1, \cdots, N\}$ is selected at random with the uniform distribution and $\gamma_n$ denotes the step-size at the $n$th iteration. This method has been analyzed in \cite{JL2019} under a choice of diminishing step-sizes and in \cite{JLZ2023} under constant step-sizes. However, the presence of stochastic gradient noise can lead the SGD iterates to oscillate dramatically and thus makes it hard to terminate the iteration properly. In order to reduce the oscillations, it is necessary to devise procedures to reduce variance of the noisy gradient. In \cite{JLZ2023} the discrepancy principle is incorporated into the choice of step-sizes leading to significant reduction of oscillations. 

In the context of large scale optimization problems 
\begin{align}\label{fso}
    \min_{x\in X} \left\{f(x):= \frac{1}{N}\sum_{i=1}^N f_i(x)\right\}
\end{align}
of finite-sum structure, where each $f_i$ is convex continuous differentiable, various variance reduction methods have been proposed to accelerate the stochastic gradient method, see \cite{DBL2014,JZ2013,LSB2012,NLST2017,SZ2013,ZMJ2013,ZX2015}. One of the most popular method is the stochastic variance reduced gradient (SVRG) method (\cite{JZ2013,ZMJ2013}) which has received tremendous attention (\cite{AY2016,HAVSKS2015,TMDQ2016,XZ2014,YLDS2021}). The SVRG method is a stochastic algorithm to solve the minimization problem (\ref{fso}) iteratively. It starts from an initial guess $x_0$ and an update frequency $m$. When a snapshot point $x_n$ is determined at the $n$th step, SVRG then calculate the full gradient $\nabla f(x_n)$ of $f$ at $x_n$ and perform $m$ steps of SGD to obtain $\{x_{n,k}: k=0, \cdots,m\}$ with $x_{n, 0} = x_n$ using the unbiased gradients
$$
g_{n,k} = \nabla f_{i_{n,k}}(x_{n,k}) - \nabla f_{i_{n,k}}(x_n) + \nabla f(x_n),
$$
where $i_{n,k}\in \{1, \cdots, N\}$ is chosen randomly via the uniform distribution. Namely, 
$$
x_{n,k+1} = x_{n,k} - \gamma g_{n,k}, \quad k = 0, \cdots, m-1
$$
with a constant step size $\gamma$. The next snapshot point $x_{n+1}$ is then defined from $x_{n,k}, k=0, \cdots, m$ in various ways, e.g. $x_{n+1}$ can be defined as the last iterate, a random choice among them, or a weighted iterate average. When SVRG is used to solve (\ref{SVRG.1}) using noisy data $y^\d$, we may consider (\ref{fso}) with $f_i(x) = \frac{1}{2} \|A_i x - y_i^\d\|^2$. Correspondingly
$$
\nabla f(x) = \frac{1}{N} A^*(A x - y^\d), \quad \nabla f_i(x) = A_i^*(A_i x- y_i^\d). 
$$
Thus, by taking the snapshot points to be the last iterates in the SVRG method, it leads to the following Algorithm \ref{alg:SVRG0} for solving linear ill-posed inverse problems which has been considered in \cite{JZZ2022} in finite-dimensions.

\begin{algorithm}
        \caption{SVRG for linear ill-posed problems \cite{JZZ2022}}\label{alg:SVRG0}
	\begin{algorithmic}[0]	
		\STATE \textbf{input:} update frequency $m$, initial guess $x_0$, and step-size $\gamma$. Set $x_0^\d := x_0$.
		
		\STATE {\bf for} $n = 0, 1, \cdots$ {\bf do}  
		
		\STATE \qquad $g_n^\d = \displaystyle{\frac{1}{N} A^* (A x_n^\d - y^\d)}$; \, $x_{n,0}^\d = x_n^\d$; 
		
		\STATE \qquad {\bf for } $k = 0, \cdots, m-1$ {\bf do} 
		
		\STATE \qquad \qquad pick $i_{n,k} \in \{1, \cdots, N\}$ randomly via uniform distribution; 
		
		\STATE \qquad \qquad $g_{n,k}^\d = A_{i_{n,k}}^*A_{i_{n,k}}(x_{n,k}^\d - x_n^\d)+g_n^\d$; 
		
		\STATE \qquad \qquad  $\displaystyle{x_{n,k+1}^\d = x_{n,k}^\d - \gamma g_{n,k}^\d}$; 
		
		\STATE \qquad {\bf end for} 
		
		\STATE \qquad $x_{n+1}^\d = x_{n,m}^\d$; 
		
		\STATE {\bf end} 
	\end{algorithmic}
\end{algorithm}

The SVRG method for large scale optimization problems of the form (\ref{fso}) has been analyzed extensively, and all the established convergence results require either the objective function $f$ to be strongly convex or the error estimates are established in terms of the objective function value. However, these results are not applicable to Algorithm \ref{alg:SVRG0} for ill-posed problems because the corresponding objective function is the residue $f(x) = \frac{1}{2N} \|A x - y^\d\|^2$ which is not strongly convex, and moreover, due to the ill-posedness of the underlying problem, the error estimate on residue does not imply any estimate on the iterates directly. Therefore, new analysis is required for understanding Algorithm \ref{alg:SVRG0} for ill-posed problems. 
Like all the other iterative regularization methods, when Algorithm \ref{alg:SVRG0} is used to solve ill-posed problems, it exhibits the semi-convergence phenomenon, i.e. the iterate tends to the sought solution at the beginning and then leaves away from the sought solution as the iteration proceeds. Thus, properly terminating the iteration is crucial for producing acceptable approximate solutions. 
Based on the bias-variance decomposition, a convergence analysis on Algorithm \ref{alg:SVRG0} has been provided in \cite{JZZ2022} by a delicate spectral theory argument when $A$ has a special structure. It has been proved that, when the sought solution $x^\dag$ satisfies the H\"{o}lder source condition $x^\dag - x_0 = (A^*A)^\nu \omega$ for some $\omega \in X$ and $\nu>0$, an order optimal error bound can be established on $x_{n_\d}^\d$ for an {\it a priori} chosen stopping index $n_\d$. 
However, the convergence analysis in \cite{JZZ2022} has the following drawbacks. 

\begin{enumerate}[leftmargin = 0.8cm]
     \item[$\bullet$] The analysis in \cite{JZZ2022} is carried out under the H\"{o}lder type source conditions on the sought solution. This type of source conditions might be too strong to be satisfied in applications. Is it possible to establish a convergence result without using source conditions?

     \item[$\bullet$] The analysis in \cite{JZZ2022} requires $m$ to be large and the step-size $\gamma>0$ to be sufficiently small, see \cite[Theorem 2.1]{JZZ2022}; no explicit formula is provided for choosing $\gamma$. The numerical simulations in \cite{JZZ2022} use $\gamma = O(1/m)$, When $m$ is chosen as $m = \rho N$ for some constant $\rho>0$ and $N$ is huge, then $\gamma$ can be very small. Using small step-sizes can slow down the convergence of the method and thus huge amount of computational time is required. Can we develop a convergence analysis which allows using larger step-sizes?

      \item[$\bullet$]  The most serious drawback is that the arguments in \cite{JZZ2022} require $A$ to have the decomposition structure $A = \Sigma V^t$, where $\Sigma$ is diagonal with nonnegative entries and $V$ is column orthonormal. Unfortunately, the forward operators arising in linear ill-posed problems seldom have this structure in general. One may argue that, by performing the singular value decomposition $A = U \Sigma V^t$, one may transform the equation $A x = y$ equivalently to $\Sigma V^t x = U^t y$ and then apply the convergence results in \cite{JZZ2022}. However, finding the singular value decomposition requires a huge amount of computational time or even is impossible if the problem size is huge. Therefore, the understanding on SVRG for linear ill-posed inverse problems is still largely open. It is desirable to have a convergence analysis without relying on the decomposition structure of $A$ as assumed in \cite{JZZ2022}. 
\end{enumerate}

In this paper we will provide a completely different novel analysis on SVRG for solving linear ill-posed problems and remove all the above mentioned drawbacks. Note that the definition of $x_{n,1}^\d$
from $x_n^\d$ in Algorithm \ref{alg:SVRG0} is a one-step of Landweber method and does not involve any stochasticity because $g_{n,0}^\d = g_n^\d$ although a random index $i_{n,0}$ is selected. This sharply contrasts to the update of $x_{n,k+1}^\d$ for $1 \le k \le m-1$ which depends heavily on the randomly selected index $i_{n, k}$. Therefore, it seems natural to modify Algorithm \ref{alg:SVRG0} by splitting these two parts and introducing two step-size parameters $\gamma_0$ and $\gamma_1$. This leads to the following Algorithm \ref{alg:SVRG} we will consider in this paper. 

\begin{algorithm}
        \caption{SVRG for linear ill-posed problems}\label{alg:SVRG}
	\begin{algorithmic}[0]	
		\STATE \textbf{input:} update frequency $m$, initial guess $x_0$, step-sizes $\gamma_0$ and $\gamma_1$. Set $x_0^\d := x_0$. 
		
		\STATE {\bf for} $n = 0, 1, \cdots$ {\bf do}  
		
		\STATE \qquad $g_n^\d = A^* (A x_n^\d - y^\d)$; 
        \STATE \qquad $x_{n,0}^\d = x_n^\d - \gamma_0 g_n^\d$; 
		
		\STATE \qquad {\bf for } $k = 0, \cdots, m-1$ {\bf do} 
		
		\STATE \qquad \qquad pick $i_{n,k} \in \{1, \cdots, N\}$ randomly via uniform distribution; 
		
		\STATE \qquad \qquad $\displaystyle{g_{n,k}^\d = A_{i_{n,k}}^*A_{i_{n,k}} (x_{n,k}^\d - x_n^\d) + \frac{1}{N} g_n^\d}$; 
		
		\STATE \qquad \qquad  $\displaystyle{x_{n,k+1}^\d = x_{n,k}^\d - \gamma_1 g_{n,k}^\d}$; 
		
		\STATE \qquad {\bf end for} 
		
		\STATE \qquad $x_{n+1}^\d = x_{n,m}^\d$; 
		
		\STATE {\bf end} 
	\end{algorithmic}
\end{algorithm}

Note that, once $x_0\in X$, $m$, $\gamma_0$ and $\gamma_1$ are fixed, the sequence $\{x_n^\d\}$ in Algorithm \ref{alg:SVRG} is completely determined by the sample path $\{i_{n,k}: n \ge 0, k = 0, \cdots, m-1\}$; changing the sample path can result in a different iterative sequence and thus 
$\{x_n^\d\}$ is a random sequence. Therefore we need to perform a stochastic analysis on Algorithm \ref{alg:SVRG}. For each integer $n \ge 0$ and $k \in \{0, \cdots, m-1\}$, let $\F_{n,k}$ denote the $\sigma$-algebra generated by the random variables $i_{n',k'}$ for $(n',k') \in \{(n',k'): 0\le n'\le n-1, 0\le k'\le m-1\}\cup \{(n, k'): 0\le k'<k\}$. Then $\{\F_{n,k}: n\ge 0 \mbox{ and } k =0, \cdots, m-1\}$ form 
a filtration which is natural to Algorithm \ref{alg:SVRG}. We will also set $\F_n := \F_{n-1, m-1}$ for $n \ge 1$. Let $\EE$ denote the expectation associated with this filtration, see \cite{B2020}. The tower property 
$$
\EE[\EE[\varphi|\F_{n,k}]] = \EE[\varphi] \quad \mbox{ for any random variable } \varphi
$$
will be frequently used. 
Our analysis on Algorithm \ref{alg:SVRG} is based on a variational approach. Without using any source condition on the sought solution $x^\dag$ we show that $\EE[\|x_{n_\d}^\d - x^\dag\|^2] \to 0$ as $\d \to 0$ if the stopping index $n_\d$ is chosen such that $n_\d \to \infty$ and $\d^2 n_\d \to 0$ as $\d \to 0$. When $x^\dag$ satisfies the benchmark source condition $x^\dag - x_0 = A^* \la^\dag$ for some $\la^\dag \in Y$, the convergence rate $\EE[\|x_{n_\d}^\d - x^\dag\|^2] = O(\d)$ holds for the stopping index $n_\d$ chosen by $n_\d \sim \d^{-1}$. Sharply contrast to \cite{JZZ2022}, our results are established for general bounded linear operator $A$, no special structure on $A$ is required. Furthermore, our analysis allows using large step sizes. In particular, our convergence results hold for 
$$
\gamma_0 = \frac{1}{\|A\|^2} \quad \mbox{ and } \quad 
\gamma_1 = \beta \min\left\{\frac{1}{L}, \frac{1}{\|A\|} \sqrt{\frac{N}{2mL}} \right\}
$$
with $0<\beta <1$, where 
\begin{align}\label{L}
 L := \max\{\|A_i\|: i=1, \cdots, N\}. 
\end{align}
In case $m = N$, both $\gamma_0$ and $\gamma_1$ are constants independent of $m$, $N$ and can be sufficiently larger than those required in \cite{JZZ2022}. Finally we also consider terminating the iteration in Algorithm \ref{alg:SVRG} by {\it a posteriori} stopping rules and demonstrate that the discrepancy principle can terminate the iterations in finite many steps almost surely. This suggests that we may incorporate the discrepancy principle into Algorithm \ref{alg:SVRG} to turn it into a practical implementable method for solving linear ill-posed inverse problems. 

This paper is organized as follows. In Section \ref{sect2} we first prove a stability result concerning Algorithm \ref{alg:SVRG}. In Section \ref{sect3} we then derive a convergence rate result when the sought solution satisfies a benchmark source condition. Based on results from Sections \ref{sect2} and \ref{sect3}, in Section \ref{sect4} we use a density argument to prove a convergence result without using any source condition. In Section \ref{sect5} we consider incorporating the discrepancy principle into Algorithm \ref{alg:SVRG} and demonstrate that the method can be terminated in finite many steps almost surely. Finally, in Section \ref{sect6} we provide various numerical results to test the performance of Algorithm \ref{alg:SVRG}. 

\section{\bf Stability estimate}\label{sect2}

Let $\{x_n^\d\}$ be defined by Algorithm \ref{alg:SVRG}. We first consider for each fixed $n$ the behavior of $x_n^\d$ as $\d \to 0$. To this end, we consider the counterpart of Algorithm \ref{alg:SVRG} with noisy data $y_i^\d$ replaced by the exact data $y_i$ and drop the superscript $\d$ in all quantities; for instance, we will write $x_n^\d$ as $x_n$, $x_{n,k}^\d$ as $x_{n,k}$ and so on. According to the definition of $x_n^\d$ and $x_n$, one can easily see that, for any fixed integer $n \ge 0$, there holds $\|x_n^\d - x_n\| \to 0$ as $\d \to 0$ along any sample path and thus $\EE[\|x_n^\d - x_n\|^2] \to 0$ as $\d \to 0$. The following result gives 
a quantitative estimate of this kind of stability.

\begin{lemma}\label{SVRG:lem6}
Let $L$ be defined by (\ref{L}). If $\gamma_0>0$ and $\gamma_1>0$ are chosen such that 
\begin{align}\label{SVRG.5}
1-\gamma_1 L >0 \quad \mbox{and} \quad 2\gamma_0 - \gamma_0^2 \|A\|^2 - \frac{2m \gamma_1^2 L}{N} > 0,
\end{align}
then  
\begin{align*}
\EE\left[\|x_n^\d - x_n\|^2\right] \le C_0 n \d^2 
\end{align*}
for all integers $n \ge 0$, where 
$$
C_0 := \frac{\gamma_0^2}{2\gamma_0 - \gamma_0^2 \|A\|^2 - 2m \gamma_1^2 L/N}
+ \frac{m \gamma_1^2}{2N(1-\gamma_1 L)}. 
$$
\end{lemma}

\begin{proof}
We use an induction argument. The result is trivial for $n=0$ because $x_0^\d = x_0$. Now we 
assume that $\EE\left[\|x_n^\d - x_n\|^2\right] \le C_0 n \d^2$ for some integer $n \ge 0$ and show the result 
for $n+1$. To see this, along any sample path we set 
$$
u_n^\d:= x_n^\d - x_n \quad \mbox{ and } \quad u_{n, k}^\d := x_{n,k}^\d - x_{n,k}. 
$$
Note that 
$$
u_{n,0}^\d  = u_n^\d - \gamma_0 A^* (A u_n^\d - y^\d + y). 
$$
Thus
\begin{align}\label{stab}
\|u_{n,0}^\d\|^2 - \|u_n^\d\|^2 
& = - 2\gamma_0 \l u_n^\d, A^*(A u_n^\d - y^\d + y)\r 
+ \gamma_0^2 \|A^*(A u_n^\d - y^\d +y)\|^2 \displaybreak[0]\nonumber \\
& = - 2\gamma_0 \l A u_n^\d, A u_n^\d - y^\d + y\r 
+ \gamma_0^2 \|A^*(A u_n^\d - y^\d +y)\|^2 \displaybreak[0]\nonumber \\
& \le - 2\gamma_0 \|A u_n^\d - y^\d + y\|^2 
- 2\gamma_0 \l y^\d - y, A u_n^\d - y^\d + y\r \nonumber \\
& \quad + \gamma_0^2 \|A\|^2 \|A u_n^\d - y^\d +y\|^2 \displaybreak[0]\nonumber \\
& \le - \left(2\gamma_0 - \gamma_0^2\|A\|^2\right) \|A u_n^\d - y^\d + y\|^2 \nonumber \\
& \quad \, + 2\gamma_0 \d \|A u_n^\d - y^\d + y\|.
\end{align}
Next, by noting that 
$$
u_{n,k+1}^\d = u_{n,k}^\d - \gamma_1 \left(A_{i_{n,k}}^* A_{i_{n,k}}(u_{n,k}^\d - u_n^\d) 
+ \frac{1}{N} A^*(A u_n^\d - y^\d + y) \right), 
$$
we therefore have  
\begin{align*}
\|u_{n,k+1}^\d\|^2 - \|u_{n,k}^\d\|^2 
& = - 2\gamma_1 \left\l u_{n,k}^\d, A_{i_{n,k}}^* A_{i_{n,k}}(u_{n,k}^\d - u_n^\d) 
+ \frac{1}{N} A^*(A u_n^\d - y^\d + y) \right\r \\
& \quad \, + \gamma_1^2 \left\|A_{i_{n,k}}^* A_{i_{n,k}}(u_{n,k}^\d - u_n^\d) 
+ \frac{1}{N} A^*(A u_n^\d - y^\d + y) \right\|^2.
\end{align*}
Consequently, by taking the expectation conditioned on $\F_{n,k}$, we can obtain 
\begin{align*}
& \EE\left[\|u_{n,k+1}^\d\|^2|\F_{n,k}\right] - \|u_{n,k}^\d\|^2 \\
& = - \frac{2\gamma_1}{N} \sum_{i=1}^N \left\l u_{n,k}^\d, A_i ^* A_i(u_{n,k}^\d - u_n^\d) 
+ \frac{1}{N} A^*(A u_n^\d - y^\d + y) \right\r \displaybreak[0] \\
& \quad \, + \frac{\gamma_1^2}{N} \sum_{i=1}^N \left\|A_i^* A_i(u_{n,k}^\d - u_n^\d) 
+ \frac{1}{N} A^*(A u_n^\d - y^\d + y) \right\|^2 \displaybreak[0] \\
& = -\frac{2\gamma_1}{N}\l u_{n,k}^\d, A^*(A u_{n,k}^\d - y^\d + y)\r \\
& \quad \, + \frac{\gamma_1^2}{N} \sum_{i=1}^N \left\|A_i^* A_i(u_{n,k}^\d - u_n^\d) 
+ \frac{1}{N} A^*(A u_n^\d - y^\d + y) \right\|^2.
\end{align*}
By virtue of the inequality $\|a+b\|^2 \le 2 (\|a\|^2 + \|b\|^2)$ and the polarization identity, we 
have 
\begin{align*}
& \EE\left[\|u_{n,k+1}^\d\|^2|\F_{n,k}\right] - \|u_{n,k}^\d\|^2 \\
& \le -\frac{2\gamma_1}{N}\l A u_{n,k}^\d, A u_{n,k}^\d - y^\d + y\r 
+ \frac{2\gamma_1^2}{N} \sum_{i=1}^N \|A_i^*(A_i u_{n,k}^\d - y_i^\d + y_i)\|^2 \displaybreak[0] \\
& \quad \, + \frac{2\gamma_1^2}{N} \sum_{i=1}^N \left\|A_i^*(A_i u_n^\d - y_i^\d + y_i)
- \frac{1}{N} A^*(A u_n^\d -y^\d + y)\right\|^2 \displaybreak[0] \\
& = - \frac{2\gamma_1}{N} \|A u_{n,k}^\d - y^\d + y\|^2
- \frac{2\gamma_1}{N} \l y^\d - y, A u_{n,k}^\d - y^\d + y\r \displaybreak[0] \\
& \quad \, + \frac{2\gamma_1^2}{N} \sum_{i=1}^N \|A_i^*(A_i u_{n, k}^\d - y_i^\d + y_i)\|^2 
+ \frac{2\gamma_1^2}{N} \sum_{i=1}^N \|A_i^*(A_i u_n^\d - y_i^\d + y_i)\|^2 \displaybreak[0] \\
& \quad \, - \frac{4 \gamma_1^2}{N^2} \sum_{i=1}^N \l A_i^*(A_i u_n^\d - y_i^\d + y_i), A^*(A u_n^\d - y^\d + y)\r \displaybreak[0] \\
& \quad \, + \frac{2\gamma_1^2}{N^2} \|A^*(A u_n^\d - y^\d + y)\|^2.
\end{align*}
By using the Cauchy-Schwarz inequality, $\|y^\d - y\| \le \d$ and the definition of $L$, we obtain 
\begin{align*}
& \EE\left[\|u_{n,k+1}^\d\|^2|\F_{n,k}\right] - \|u_{n,k}^\d\|^2 \\
& \le  - \frac{2\gamma_1}{N} \|A u_{n,k}^\d - y^\d + y\|^2
+ \frac{2\gamma_1}{N} \d\|A u_{n,k}^\d - y^\d + y\| \displaybreak[0] \\
& \quad \, + \frac{2\gamma_1^2}{N} \sum_{i=1}^N \|A_i\|^2 \|A_i u_{n, k}^\d - y_i^\d + y_i\|^2 
+ \frac{2\gamma_1^2}{N} \sum_{i=1}^N \|A_i\|^2 \|A_i u_n^\d - y_i^\d + y_i\|^2 \displaybreak[0] \\
& \quad \, - \frac{2\gamma_1^2}{N^2} \|A^*(A u_n^\d - y^\d + y)\|^2 \displaybreak[0] \\
& \le  - \frac{2\gamma_1 (1-\gamma_1 L)}{N} \|A u_{n,k}^\d - y^\d + y\|^2 
+ \frac{2\gamma_1}{N} \d\|A u_{n,k}^\d - y^\d + y\| \\
& \quad \, + \frac{2\gamma_1^2 L}{N} \|A u_n^\d - y^\d + y\|^2. 
\end{align*}
In view of the inequality 
$$
\frac{2\gamma_1}{N} \d \|A u_{n, k}^\d - y^\d +y\| 
\le \frac{2\gamma_1 (1-\gamma_1 L)}{N} \|A u_{n,k}^\d - y^\d + y\|^2 
+ \frac{\gamma_1}{2N (1-\gamma_1 L)} \d^2, 
$$
we further obtain 
\begin{align*}
& \EE\left[\|u_{n,k+1}^\d\|^2|\F_{n,k}\right] - \|u_{n,k}^\d\|^2 
\le  \frac{\gamma_1 \d^2}{2N (1-\gamma_1 L)} + \frac{2\gamma_1^2 L}{N} \|A u_n^\d - y^\d + y\|^2.
\end{align*}
Consequently 
\begin{align*}
& \EE\left[\|u_{n,k+1}^\d\|^2|\F_n\right] - \EE\left[\|u_{n,k}^\d\|^2|\F_n\right] 
\le  \frac{\gamma_1 \d^2}{2N (1-\gamma_1 L)} + \frac{2\gamma_1^2 L}{N} \|A u_n^\d - y^\d + y\|^2
\end{align*}
and by summing over $k$ from $k=0$ to $k = m-1$ we obtain 
\begin{align*}
& \EE\left[\|u_{n+1}^\d\|^2|\F_n\right] - \|u_{n,0}^\d\|^2 
\le  \frac{m\gamma_1 \d^2}{2N (1-\gamma_1 L)} + \frac{2m\gamma_1^2 L}{N} \|A u_n^\d - y^\d + y\|^2.
\end{align*}
Combining this with (\ref{stab}) shows that 
\begin{align*}
\EE\left[\|u_{n+1}^\d\|^2|\F_n\right] - \|u_n^\d\|^2 
& \le  - \left(2\gamma_0 - \gamma_0^2 \|A\|^2 - \frac{2m \gamma_1^2 L}{N}\right) 
\|A u_n^\d - y^\d + y\|^2 \\
& \quad \, + 2\gamma_0 \d \|A u_n^\d - y^\d + y\| 
+ \frac{m\gamma_1 \d^2}{2N (1-\gamma_1 L)}.
\end{align*}
By using the inequality 
\begin{align*}
2\gamma_0 \d \|A u_n^\d - y^\d + y\|
& \le \left(2\gamma_0 - \gamma_0^2 \|A\|^2 - \frac{2m \gamma_1^2 L}{N}\right) 
\|A u_n^\d - y^\d + y\|^2 \\
&\quad \, + \frac{\gamma_0^2 \d^2}{2\gamma_0 - \gamma_0^2 \|A\|^2 - 2m \gamma_1^2 L/N}
\end{align*}
we can conclude 
\begin{align*}
& \EE\left[\|u_{n+1}^\d\|^2|\F_n\right] - \|u_n^\d\|^2 \\
& \le \left(\frac{\gamma_0^2}{2\gamma_0 - \gamma_0^2 \|A\|^2 - 2m \gamma_1^2 L/N} 
+ \frac{m\gamma_1}{2N (1-\gamma_1 L)}\right) \d^2 = C_0 \d^2.
\end{align*}
By taking the full expectation and using the induction hypothesis, we obtain 
$$
\EE\left[\|u_{n+1}^\d\|^2\right] \le \EE\left[\|u_n^\d\|^2\right] + C_0\d^2 \le C_0(n+1)\d^2.
$$
This completes the proof. 
\end{proof}

\begin{remark}\label{remark_stepsize}
One can easily see that (\ref{SVRG.5}) holds if we choose $\gamma_0$ and $\gamma_1$ such that 
\begin{equation}
    \gamma_0 = \frac{\a}{\|A\|^2} \quad \mbox{ and } \quad 
    \gamma_1 = \beta \min\left\{\frac{1}{L}, \frac{1}{\|A\|} \sqrt{\frac{(2-\a)\a N}{2m L}}\right\}
    \label{step-size}
\end{equation}
for some $0<\a<2$ and $0<\beta<1$. 
\end{remark}

\section{\bf Rate of convergence}\label{sect3}

In this section we will derive the error estimate on $\EE[\|x_n^\d - x^\dag\|^2]$ under the 
benchmark source condition
\begin{align}\label{sc}
x^\dag - x_0 = A^* \la^\dag \quad \mbox{ for some } \la^\dag \in Y. 
\end{align}
According to Lemma \ref{SVRG:lem6}, we need only to estimate $\EE[|\|x_n - x^\dag\|^2]$. 
To this end, we start proving the following result. 

\begin{lemma}\label{SVRG.lem1}
Consider Algorithm \ref{alg:SVRG}, Assume $\gamma_0>0$ and $\gamma_1>0$ satisfy (\ref{SVRG.5}).
Then there holds  
\begin{align*}
\EE\left[\|x_{n+1} - x^\dag\|^2 | \F_n\right] - \|x_n - x^\dag\|^2 
& \le -\frac{2 \gamma_1 (1-\gamma_1 L)}{N} \sum_{k=0}^{m-1} \EE\left[\|A x_{n,k} - y\|^2 |\F_n\right] \nonumber \\
& \quad \, - \left(2 \gamma_0 - \gamma_0^2 \|A\|^2 -  \frac{2 m \gamma_1^2 L}{N}\right) \|A x_n - y\|^2
\end{align*}
for all integers $n\ge 0$. 
\end{lemma}

\begin{proof}
We first use the polarization identity and the definition of $x_{n,k+1}$ to write  
\begin{align*}
& \|x_{n,k+1} - x^\dag\|^2 - \|x_{n,k} - x^\dag\|^2 \\
& = 2 \l x_{n,k+1} - x_{n,k}, x_{n,k} - x^\dag\r + \|x_{n, k+1} - x_{n,k}\|^2  \displaybreak[0]\\
& = - 2 \gamma_1 \left\l A_{i_{n,k}}^*A_{i_{n,k}}(x_{n,k} -x_n) + \frac{1}{N} A^*(A x_n -y), x_{n,k} - x^\dag\right\r  \displaybreak[0]\\
& \quad\, + \gamma_1^2 \left\|A_{i_{n,k}}^*A_{i_{n,k}}(x_{n,k} -x_n) + \frac{1}{N} A^*(A x_n -y) \right\|^2. 
\end{align*}
Taking the conditional expectation on $\F_{n,k}$ gives 
\begin{align*}
& \EE\left[\|x_{n,k+1} - x^\dag\|^2 | \F_{n,k}\right] - \|x_{n,k} - x^\dag\|^2  \displaybreak[0] \\
&  = - \frac{2\gamma_1}{N} \sum_{i=1}^N \left\l A_i^*A_i(x_{n,k} -x_n) + \frac{1}{N} A^*(A x_n -y), x_{n,k} - x^\dag\right\r  \displaybreak[0]\\
& \quad\, + \frac{\gamma_1^2}{N} \sum_{i=1}^N \left\|A_i^*A_i(x_{n,k} -x_n) + \frac{1}{N} A^*(A x_n -y) \right\|^2  \displaybreak[0]\\
&  = - \frac{2\gamma_1}{N} \l A^*(A x_{n,k}-y), x_{n,k} - x^\dag \r  \displaybreak[0]\\
& \quad\, + \frac{\gamma_1^2}{N} \sum_{i=1}^N \left\|A_i^*A_i(x_{n,k} -x_n) + \frac{1}{N} A^*(A x_n -y) \right\|^2.
\end{align*}
By using the inequality $\|a+b\|^2 \le 2(\|a\|^2 + \|b\|^2)$ and the polarization identity, we have 
\begin{align*}
& \EE\left[\|x_{n,k+1} - x^\dag\|^2 | \F_{n,k}\right] - \|x_{n,k} - x^\dag\|^2  \displaybreak[0] \\
&  \le - \frac{2\gamma_1}{N} \|A x_{n,k} - y\|^2
+ \frac{2\gamma_1^2}{N} \sum_{i=1}^N \|A_i^*(A_i x_{n,k} -y_i)\|^2  \displaybreak[0]\\
& \quad \, + \frac{2\gamma_1^2}{N} \sum_{i=1}^N \left\|A_i^*(A_i x_n - y_i) - \frac{1}{N} A^*(A x_n - y)\right\|^2 \displaybreak[0]\\
& = - \frac{2\gamma_1}{N} \|A x_{n,k} - y\|^2 
+ \frac{2\gamma_1^2}{N} \sum_{i=1}^N \|A_i^*(A_i x_{n,k} - y_i)\|^2 
+ \frac{2\gamma_1^2}{N} \sum_{i=1}^N \|A_i^*(A_i x_n - y_i)\|^2 \displaybreak[0]\\
& \quad \, - \frac{4\gamma_1^2}{N^2} \sum_{i=1}^N \l A_i^*(A_i x_n - y_i), A^*(A x_n - y)\r 
+ \frac{2\gamma_1^2}{N^2} \|A^*(A x_n - y)\|^2 \displaybreak[0]\\
& \le - \frac{2\gamma_1}{N} \|A x_{n,k} - y\|^2 
+ \frac{2\gamma_1^2}{N} \sum_{i=1}^N \|A_i\|^2 \left(\|A_i x_{n,k} - y_i\|^2 
+ \|A_i x_n - y_i\|^2\right) \displaybreak[0] \\
& \quad \, - \frac{2\gamma_1^2}{N^2} \|A^*(A x_n - y)\|^2 \displaybreak[0]\\
& \le - \frac{1}{N}\left(2 \gamma_1 - 2 \gamma_1^2 L\right) \| A x_{n,k} -y\|^2 
+ \frac{2 \gamma_1^2 L}{N} \|A x_n - y\|^2.
\end{align*}
Consequently, by using the tower property of conditional expectation, we can obtain 
\begin{align*}
& \EE\left[\|x_{n,k+1} - x^\dag\|^2 | \F_n\right] - \EE\left[\|x_{n,k} - x^\dag\|^2 | \F_n\right]  \displaybreak[0]\\
& \le - \frac{1}{N}\left(2 \gamma_1 - 2 \gamma_1^2 L\right) \EE\left[\|A x_{n,k} - y\|^2|\F_n\right] 
+ \frac{2 \gamma_1^2 L}{N} \|A x_n - y\|^2.
\end{align*}
Summing this inequality over $k$ from $0$ to $m-1$, noting that $x_{n+1} = x_{n,m}$ and 
$\EE[\|x_{n,0} - x^\dag\|^2 | \F_n] = \|x_{n,0} - x^\dag\|^2$, we can obtain 
\begin{align}\label{SVRG.3}
\EE\left[\|x_{n+1} - x^\dag\|^2 | \F_n\right] - \|x_{n,0} - x^\dag\|^2 
& \le -\frac{2 \gamma_1 (1-\gamma_1 L)}{N} \sum_{k=0}^{m-1} \EE\left[\|A x_{n,k} - y\|^2|\F_n\right] \nonumber \\
& \quad \, + \frac{2 m \gamma_1^2 L}{N} \|A x_n - y\|^2. 
\end{align}
Next, by the definition of $x_{n,0}$, we have 
\begin{align*}
\|x_{n,0} - x^\dag\|^2 - \|x_n - x^\dag\|^2 
& = 2\l x_{n,0} - x_n, x_n - x^\dag\r + \|x_{n,0} - x_n\|^2 \\
& = - 2\gamma_0 \l A^*(A x_n - y), x_n - x^\dag\r + \gamma_0^2 \|A^*(A x_n - y)\|^2 \\
& \le -\left(2 \gamma_0 - \gamma_0^2 \|A\|^2\right) \|A x_n -y\|^2. 
\end{align*}
Adding this inequality to (\ref{SVRG.3}), we therefore complete the proof. 
\end{proof}

To proceed further, we need an equivalent formulation of Algorithm \ref{alg:SVRG} with exact data. 
From the definition of Algorithm \ref{alg:SVRG} we can note that 
$x_n, x_{n,k}\in x_0 + \mbox{Ran}(A^*)$ and thus there exist $\la_n, \la_{n,k} \in Y$ 
such that $x_n = x_0 + A^* \la_n$ and $x_{n,k} = x_0 + A^* \la_{n,k}$. We need a procedure 
to construct such $\la_n$ and $\la_{n,k}$ and then use them to achieve our goal. This 
inspires us to introduce the following Algorithm \ref{alg:SVRG2} which is easily 
seen to be equivalent to Algorithm \ref{alg:SVRG} with exact data, i.e. the random sequences 
$\{x_n\}$ and $\{x_{n,k}\}$ produced by Algorithm \ref{alg:SVRG2} are exactly the same 
ones produced by Algorithm \ref{alg:SVRG} with exact data. 

\begin{algorithm}
        \caption{}\label{alg:SVRG2}
	\begin{algorithmic}[0]	
		\STATE \textbf{input:} update frequency $m$, initial guess $\la_0 = 0\in Y$, $x_0\in X$, step-sizes $\gamma_0$, $\gamma_1$. 
		
		\STATE {\bf for} $n = 0, 1, \cdots$ {\bf do}  
		
		\STATE \qquad $\displaystyle{\mu_n = A x_n - y}$; 
  
        \STATE \qquad $\displaystyle{\la_{n,0} = \la_n - \gamma_0 \mu_n}$; \, \, $\displaystyle{x_{n,0} = x_0 + A^* \la_{n,0}}$; 
		
		\STATE \qquad {\bf for } $k = 0, \cdots, m-1$ {\bf do} 
		
		\STATE \qquad \qquad pick $i_{n,k} \in \{1, \cdots, N\}$ randomly via uniform distribution; 
		
		\STATE \qquad \qquad $\displaystyle{\mu_{n,k} = (0, \cdots, 0, A_{i_{n,k}} (x_{n,k} - x_n), 0, \cdots, 0)  + \frac{1}{N}\mu_n}$; 
		
		\STATE \qquad \qquad  $\displaystyle{\la_{n,k+1} = \la_{n,k} - \gamma_1 \mu_{n,k}}$; \, \, $x_{n,k+1} = x_0 + A^*\la_{n, k+1}$; 
		
		\STATE \qquad {\bf end for} 
		
		\STATE \qquad $\la_{n+1} = \la_{n,m}$; \, \, $x_{n+1} = x_0 +  A^* \la_{n+1}$; 
		
		\STATE {\bf end} 
	\end{algorithmic}
\end{algorithm}

In the formulation of Algorithm \ref{alg:SVRG2}, $(0, \cdots, 0, A_{i_{n,k}} (x_{n,k} - x_n), 0, \cdots, 0)$ denotes the element in $Y$ whose $i_{n,k}$-th component is $A_{i_{n,k}} (x_{n,k} - x_n)$ and other components are $0$. 

\begin{lemma}\label{SVRG.lem4}
Assume the source condition (\ref{sc}) holds. For any integer $n \ge 0$ there holds 
\begin{align*}
\EE[\|\la_{n+1} - \la^\dag\|^2] - \EE[\|\la_n - \la^\dag\|^2] 
& \le -2\gamma_0 \EE[\|x_n - x^\dag\|^2] - \frac{2\gamma_1}{N} \sum_{k=0}^{m-1} \EE[\|x_{n,k} - x^\dag\|^2] \\
& \quad \, + \frac{2\gamma_1^2}{N} \sum_{k=0}^{m-1} \EE[\|A x_{n,k} -y\|^2] \\
& \quad \, + \left (\gamma_0^2 + \frac{2m\gamma_1^2}{N}\right) \EE[\|A x_n - y\|^2] 
\end{align*}
\end{lemma}

\begin{proof}
By the definition of $\la_{n,0}$, $x_n = x_0 + A^* \la_n$, and (\ref{sc}) we first have 
\begin{align}\label{SVRG.10}
\|\la_{n,0} - \la^\dag\|^2 - \|\la_n - \la^\dag\|^2 
& = 2 \l \la_{n,0} - \la_n, \la_n - \la^\dag\r + \|\la_{n,0} - \la_n\|^2 \nonumber \\
& = - 2\gamma_0 \l A x_n - y, \la_n - \la^\dag\r 
+ \gamma_0^2 \|A x_n - y\|^2 \nonumber \\
& = - 2\gamma_0 \l x_n - x^\dag, A^*(\la_n - \la^\dag)\r 
+ \gamma_0^2 \|A x_n - y\|^2 \nonumber \\
& = - 2\gamma_0 \|x_n - x^\dag\|^2 
+ \gamma_0^2 \|A x_n - y\|^2.
\end{align}
Next by using the definition of $\la_{n, k+1}$ we have 
\begin{align*}
\|\la_{n, k+1} - \la^\dag\|^2 - \|\la_{n,k} -\la^\dag\|^2 
& = 2 \l \la_{n, k+1} - \la_{n,k}, \la_{n,k} - \la^\dag\r + \|\la_{n,k+1} - \la_{n,k}\|^2 \displaybreak[0]\\
& = - 2\gamma_1 \l \mu_{n,k}, \la_{n,k} - \la^\dag\r + \gamma_1^2 \|\mu_{n,k}\|^2 \displaybreak[0]\\
& = -2\gamma_1 \l A_{i_{n,k}}(x_{n,k} - x_n), (\la_{n,k} - \la^\dag)_{i_{n,k}}\r \displaybreak[0]\\
& \quad \, - \frac{2 \gamma_1}{N} \l \mu_n, \la_{n,k} - \la^\dag\r + \gamma_1^2 \|\mu_{n,k}\|^2,
\end{align*}
where we used $(\la_{n,k} - \la^\dag)_i$ to denote the $i$th component of $\la_{n,k} - \la^\dag$. 
Therefore, by taking the conditional expectation on $\F_{n,k}$ and using $x_{n,k} = x_0 + A^*\la_{n,k}$ and (\ref{sc}), we can obtain
\begin{align*}
&\EE[\|\la_{n, k+1} - \la^\dag\|^2|\F_{n,k}] - \|\la_{n,k} -\la^\dag\|^2 \displaybreak[0]\\
& = -\frac{2\gamma_1}{N} \sum_{i=1}^N \l A_i(x_{n,k} - x_n), (\la_{n,k} - \la^\dag)_i\r 
- \frac{2 \gamma_1}{N} \l \mu_n, \la_{n,k} - \la^\dag\r + \gamma_1^2 \EE[\|\mu_{n,k}\|^2|\F_{n,k}] \displaybreak[0]\\
& = - \frac{2\gamma_1}{N} \l A(x_{n,k} - x_n), \la_{n,k} - \la^\dag\r 
-\frac{2\gamma_1}{N} \l A x_n -y, \la_{n,k} - \la^\dag\r 
+ \gamma_1^2 \EE[\|\mu_{n,k}\|^2|\F_{n,k}] \displaybreak[0]\\
& = - \frac{2\gamma_1}{N} \l A x_{n,k} - y, \la_{n,k} - \la^\dag\r 
+ \gamma_1^2 \EE[\|\mu_{n,k}\|^2|\F_{n,k}] \displaybreak[0]\\
& = - \frac{2\gamma_1}{N} \l x_{n,k} - x^\dag, A^*(\la_{n,k} - \la^\dag)\r 
+ \gamma_1^2 \EE[\|\mu_{n,k}\|^2|\F_{n,k}] \displaybreak[0]\\
& = - \frac{2\gamma_1}{N} \|x_{n,k} - x^\dag\|^2 
+ \gamma_1^2 \EE[\|\mu_{n,k}\|^2|\F_{n,k}]. 
\end{align*}
We need to estimate $\EE[\|\mu_{n,k}\|^2|\F_{n,k}]$. By the definition of $\mu_{n,k}$ we have 
\begin{align*}
\|\mu_{n,k}\|^2 
& = \frac{1}{N^2} \sum_{j\ne i_{n,k}} \|A_j x_n - y_j\|^2 \\
& \quad \, + \left\| (A_{i_{n,k}} x_{n,k} - y_{i_{n,k}}) - \frac{N-1}{N} (A_{i_{n,k}} x_n - y_{i_{n,k}})\right\|^2. 
\end{align*}
Thus 
\begin{align*}
& \EE[\|\mu_{n,k}\|^2 | \F_{n,k}] \\
& = \frac{1}{N^3} \sum_{i=1}^N \sum_{j\ne i} \|A_j x_n - y_j\|^2 
+ \frac{1}{N} \sum_{i=1}^N \left\|(A_i x_{n,k} - y_i) - \frac{N-1}{N} (A_i x_n - y_i)\right\|^2 \\
& \le \frac{N-1}{N^3} \sum_{i=1}^N \|A_i x_n - y_i\|^2 + \frac{2}{N} \sum_{i=1}^N \|A_i x_{n,k} - y_i\|^2 \\
& \quad \, + \frac{2}{N} \left(\frac{N-1}{N}\right)^2 \sum_{i=1}^N \|A_i x_n - y_i\|^2 \\
& \le \frac{2}{N} \|A x_n - y\|^2 + \frac{2}{N} \|A x_{n,k} - y\|^2 
\end{align*}
Consequently 
\begin{align*}
&\EE[\|\la_{n,k+1} - \la^\dag\|^2|\F_{n,k}] - \|\la_{n,k} - \la^\dag\|^2 \\
& \le -\frac{2\gamma_1}{N} \|x_{n,k} - x^\dag\|^2 + \frac{2\gamma_1^2}{N} \|A x_{n,k} - y\|^2 + \frac{2\gamma_1^2}{N} \|A x_n -y\|^2
\end{align*}
Therefore 
\begin{align*}
&\EE[\|\la_{n,k+1} - \la^\dag\|^2|\F_n] - \EE[\|\la_{n,k} - \la^\dag\|^2|\F_n] \\
& \le -\frac{2\gamma_1}{N} \EE[\|x_{n,k} - x^\dag\|^2|\F_n] 
+ \frac{2\gamma_1^2}{N} \EE[\|A x_{n,k} - y\|^2|\F_n] + \frac{2\gamma_1^2}{N} \|A x_n -y\|^2
\end{align*}    
Summing this inequality over $k$ from $k = 0$ to $k = m-1$ and then adding with (\ref{SVRG.10})
we thus obtain 
\begin{align*}
\EE[\|\la_{n+1} - \la^\dag\|^2 | \F_n] - \|\la_n - \la^\dag\|^2 
& \le - 2\gamma_0 \|x_n - x^\dag\|^2 - \frac{2\gamma_1}{N} \sum_{k=0}^{m-1} \EE[\|x_{n,k} - x^\dag\|^2|\F_n] \displaybreak[0]\\
& \quad \, + \frac{2\gamma_1^2}{N} \sum_{k=0}^{m-1} \EE[\|A x_{n,k} -y\|^2|\F_n] \\
& \quad \, + \left(\gamma_0^2 + \frac{2m\gamma_1^2}{N}\right) \|A x_n - y\|^2 
\end{align*}
which implies the desired result by taking the full expectation. 
\end{proof}

\begin{lemma}\label{SVRG:lem5}
Consider Algorithm \ref{alg:SVRG} with exact data and assume that $\gamma_0>0$ and $\gamma_1>0$ are 
chosen such that (\ref{SVRG.5}) holds. If $x^\dag$ satisfies the source condition (\ref{sc}), then 
$$
\EE[\|x_n - x^\dag\|^2] \le \frac{\|x_0 - x^\dag\|^2 + \eta \|\la^\dag\|^2}{2 \gamma_0 \eta (n+1)}
$$
for all integers $n \ge 0$, where 
$$
\eta:= \min\left\{\frac{1-\gamma_1 L}{\gamma_1}, 
\frac{2\gamma_0 - \gamma_0^2 \|A\|^2 - 2m \gamma_1^2 L/N}{\gamma_0^2 + 2m \gamma_1^2/N}\right\}.
$$
\end{lemma}

\begin{proof}
According to Lemma \ref{SVRG.lem1} we have 
\begin{align}\label{SVRG.9}
\EE\left[\|x_{n+1} - x^\dag\|^2\right] - \EE[\|x_n - x^\dag\|^2] 
& \le - \left(2 \gamma_0 - \gamma_0^2 \|A\|^2 -  \frac{2 m \gamma_1^2 L}{N}\right) 
\EE[\|A x_n - y\|^2] \nonumber \\
& \quad \, -\frac{2 \gamma_1 (1-\gamma_1 L)}{N} \sum_{k=0}^{m-1} \EE\left[\|A x_{n,k} - y\|^2\right].
\end{align}
Consider the sequence 
$$
\Delta_n:= \|x_n-x^\dag\|^2 + \eta \|\la_n - \la^\dag\|^2, \quad n = 0, 1, \cdots.
$$
It follows from (\ref{SVRG.9}) and Lemma \ref{SVRG.lem4}  that 
\begin{align*}
\EE[\Delta_{n+1}] - \EE[\Delta_n] 
& \le - 2\gamma_0 \eta \EE[\|x_n-x^\dag\|^2] 
- \frac{2\gamma_1 \eta}{N} \sum_{k=0}^m \EE[\|x_{n,k} - x^\dag\|^2] \\
& \le - 2\gamma_0 \eta \EE[\|x_n-x^\dag\|^2]. 
\end{align*}
Consequently 
$$
2\gamma_0 \eta \sum_{l=0}^n \EE[\|x_l - x^\dag\|^2] \le \EE[\Delta_0] = \Delta_0. 
$$
Since (\ref{SVRG.9}) and (\ref{SVRG.5}) imply that $\EE[\|x_l - x^\dag\|]$ is monotonically decreasing, we therefore have 
$$
2\gamma_0 \eta (n+1) \EE[\|x_n - x^\dag\|^2] \le \Delta_0 
$$
which implies the desired result. 
\end{proof}

\begin{theorem}\label{svrg.thm}
Consider Algorithm \ref{alg:SVRG} with $\gamma_0>0$ and $\gamma_1>0$ being chosen such that 
(\ref{SVRG.5}) holds. Assume that $x^\dag$ satisfies the source condition (\ref{sc}). If
the integer $n_\d$ is chosen such that $n_\d \sim \d^{-1}$, then 
$$
\EE\left[\|x_{n_\d}^\d - x^\dag\|^2\right] \le C_1 \d,
$$
where $C_1$ is a constant depending only on $\gamma_0$, $\gamma_1$, $\|A\|$, $L$, the ratio 
$m/N$, $\|x_0-x^\dag\|$ and $\|\la^\dag\|$. 
\end{theorem}

\begin{proof}
By the triangle inequality we have 
$$
\|x_n^\d - x^\dag\|^2 \le (\|x_n^\d- x_n\| + \|x_n - x^\dag\|)^2 
\le 2 \|x_n^\d - x_n\|^2 + 2 \|x_n - x^\dag\|^2. 
$$
Thus 
$$
\EE\left[||x_n^\d - x^\dag\|^2\right] \le 2 \EE\left[\|x_n^\d - x_n\|^2\right] 
+ 2 \EE\left[\|x_n - x^\dag\|^2\right] 
$$
From Lemma \ref{SVRG:lem5} and Lemma \ref{SVRG:lem6} it then follows  that 
\begin{align*}
\EE\left[\|x_n^\d - x^\dag\|^2\right] 
\le \frac{\|x_0-x^\dag\|^2 + \eta \|\la^\dag\|^2}{2\gamma_0 \eta (n+1)}  + C_0 n \d^2
\end{align*}
for all integers $n \ge 0$. With the choice $n_\d \sim \d^{-1}$ we thus obtain the desired 
convergence rate. 
\end{proof}

\section{\bf Convergence}\label{sect4}

In Theorem \ref{svrg.thm} we have established a convergence rate result for Algorithm \ref{alg:SVRG}  
when the $x_0$-minimal norm solution $x^\dag$ satisfies the source condition (\ref{sc}). This source 
condition might be too strong to be satisfied in applications. It is necessary to establish 
a convergence result on Algorithm \ref{alg:SVRG} without using any source condition on $x^\dag$. 
Considering the stability estimate given in Lemma \ref{SVRG:lem6}, we will achieve the goal by 
showing that $\EE[\|x_n - x^\dag\|^2] \to 0$ as $n \to \infty$. We will 
use a perturbation argument developed in \cite{Jin2010,Jin2011}. Namely, as an $x_0$-minimal 
norm solution, there holds $x^\dag - x_0 \in \mbox{Null}(A)^\perp = \overline{\mbox{Ran}(A^*)}$, 
and thus we may choose $\hat x_0 \in X$ as close to $x_0$ as we want such that 
$x^\dag - \hat x_0 \in \mbox{Ran}(A^*)$. We then define $\{\hat x_n, \hat x_{n,k}\}$ by Algorithm 
\ref{alg:SVRG} with exact data and with the initial guess $x_0$ replaced by $\hat x_0$. 
We will establish $\EE[\|x_n - x^\dag\|^2] \to 0$ as $n\to \infty$ by deriving estimates on $\EE[\|\hat x_n - x^\dag\|^2]$ and $\EE[\|x_n - \hat x_n\|^2]$. 

For $\EE[\|\hat x_n - x^\dag\|^2]$ we can apply the same argument in the proof of 
Lemma \ref{SVRG:lem5} to the sequence $\{\hat x_n\}$ to obtain the following result.

\begin{lemma}\label{SVRG:lem21}
    Consider the sequence $\{\hat x_n\}$ defined by Algorithm \ref{alg:SVRG} with exact data and with 
    $x_0$ replaced by $\hat x_0$, where $\hat x_0$ is chosen such that $x^\dag - \hat x_0 \in \emph{Ran}(A^*)$. Assume that $\gamma_0>0$ and $\gamma_1>0$ are chosen such that (\ref{SVRG.5}) holds. Then for any integer $n\ge 0$ there holds 
    \begin{align*}
        \EE[\|\hat x_n - x^\dag\|^2]
        \le \frac{\|\hat x_0 - x^\dag\|^2 + \eta \|\hat \la^\dag\|^2}{2 \gamma_0 \eta (n+1)}, 
    \end{align*}
    where $\eta>0$ is the constant defined in Lemma \ref{SVRG:lem5} and $\hat \la^\dag \in Y$ is such 
    that $x^\dag - \hat x_0 = A^* \hat \la^\dag$. 
\end{lemma}

We next derive estimate on $\EE[\|x_n - \hat x_n\|^2]$ in terms of $\|x_0 - \hat x_0\|^2$.  
We have the following stability result on $x_n$ with respect to the perturbation of the initial 
guess $x_0$.

\begin{lemma}\label{SVRG:lem22}
    Assume that $\gamma_0>0$ and $\gamma_1>0$ are chosen such that (\ref{SVRG.5}) is satisfied. Then there holds 
    $$
         \EE[\|x_n - \hat x_n\|^2] \le \|x_0 - \hat x_0\|^2
    $$ 
    for all integers $n\ge 0$.
\end{lemma}

\begin{proof}
    Let $z_n:=x_n - \hat x_n$ and $z_{n,k} = x_{n,k} - \hat x_{n,k}$. Then, by the definition of 
    $\{x_n, x_{n,k}\}$ and $\{\hat x_n, \hat x_{n,k}\}$ we have 
    \begin{align*}
        z_{n,0} = z_n - \gamma_0 A^*A z_n
    \end{align*}
    and 
    \begin{align*}
        z_{n, k+1} = z_{n,k} - \gamma_1 A_{i_{n,k}}^* A_{i_{n,k}}(z_{n,k} - z_n) - \frac{\gamma_1}{N} A^*A z_n
    \end{align*}
    for all $n = 0, 1, \cdots$ and $k = 0, \cdots, m-1$. Therefore
    \begin{align}\label{svrg.211}
        \|z_{n,0}\|^2 & = \|z_n\|^2 - 2 \gamma_0\l z_n, A^*A z_n\r + \gamma_0^2 \|A^*Az_n\|^2 \nonumber \\
        & = \|z_n\|^2 - 2\gamma_0 \|Az_n\|^2 + \gamma_0^2 \|A^*Az_n\|^2 \nonumber \\
        & \le \|z_n\|^2 - (2\gamma_0-\gamma_0^2\|A\|^2) \|Az_n\|^2
    \end{align}
    and 
    \begin{align*}
        \|z_{n,k+1}\|^2 - \|z_{n,k}\|^2 
        & = -2\gamma_1 \left\l z_{n,k}, A_{i_{n,k}}^*A_{i_{n,k}}(z_{n,k}-z_n) + \frac{1}{N}A^*Az_n\right\r\\
        & \quad\, +\gamma_1^2 \left\|A_{i_{n,k}}^*A_{i_{n,k}}(z_{n,k}-z_n) + \frac{1}{N}A^*Az_n\right\|^2.
    \end{align*}
    Consequently
    \begin{align*}
        \EE[\|z_{n,k+1}\|^2|\F_{n,k}] - \|z_{n,k}\|^2
        & =  -\frac{2\gamma_1}{N}\sum_{i=1}^N \left\l z_{n,k}, A_i^*A_i(z_{n,k}-z_n) + \frac{1}{N}A^*Az_n\right\r  \displaybreak[0]\\
        & \quad\, + \frac{\gamma_1^2}{N} \sum_{i=1}^N \left\|A_i^*A_i(z_{n,k}-z_n) + \frac{1}{N}A^*Az_n\right\|^2  \displaybreak[0]\\
        & =  -\frac{2\gamma_1}{N} \left\l z_{n,k}, A^* A z_{n,k} \right\r\\
        & \quad\, + \frac{\gamma_1^2}{N} \sum_{i=1}^N \left\|A_i^*A_i (z_{n,k}-z_n) + \frac{1}{N}A^*Az_n\right\|^2.
    \end{align*}
    By using the inequality $\|a+b\|^2 \le 2(\|a\|^2 + \|b\|^2)$, we further have 
    \begin{align*}
        \EE[\|z_{n,k+1}\|^2|\F_{n,k}] - \|z_{n,k}\|^2
        & \le -\frac{2\gamma_1}{N} \|A z_{n,k}\|^2 
        + \frac{2\gamma_1^2}{N} \sum_{i=1}^N \|A_i^*A_iz_{n,k}\|^2  \displaybreak[0]\\
        & \quad\, + \frac{2\gamma_1^2}{N} \sum_{i=1}^N \left\|A_i^*A_i z_n - \frac{1}{N}A^*A z_n\right\|^2 \displaybreak[0]\\
        & = - \frac{2\gamma_1}{N} \|A z_{n,k}\|^2 
        + \frac{2\gamma_1^2}{N} \sum_{i=1}^N \|A_i^*A_iz_{n,k}\|^2  \displaybreak[0]\\
        & \quad\, + \frac{2\gamma_1^2}{N} \sum_{i=1}^N \|A_i^*A_i z_n\|^2 
        + \frac{2\gamma_1^2}{N^2} \|A^*A z_n\|^2  \displaybreak[0]\\
        & \quad \, - \frac{4\gamma_1^2}{N^2} \sum_{i=1}^N \l A_i^*A_i z_n, A^*A z_n \r \displaybreak[0]\\
        & \le -\frac{2\gamma_1}{N} \|A z_{n,k}\|^2 
        + \frac{2\gamma_1^2}{N} \sum_{i=1}^N \|A_i\|^2 \|A_iz_{n,k}\|^2  \displaybreak[0]\\
        & \quad\, + \frac{2\gamma_1^2}{N} \sum_{i=1}^N \|A_i\|^2 \|A_i z_n\|^2 
        - \frac{2\gamma_1^2}{N^2} \|A^*A z_n\|^2.
    \end{align*}
    By the definition of $L$ and $1-\gamma_1 L>0$ we then have 
    \begin{align*}
        &\EE[\|z_{n,k+1}\|^2|\F_{n,k}] - \|z_{n,k}\|^2 \\
        & \le - \frac{2\gamma_1(1-\gamma_1 L)}{N} \|A z_{n,k}\|^2 
        + \frac{2\gamma_1^2 L}{N} \|A z_n\|^2
        \le \frac{2\gamma_1^2 L}{N} \|A z_n\|^2.
    \end{align*}
    Therefore
     \begin{align*}
        \EE[\|z_{n,k+1}\|^2|\F_n] - \EE[\|z_{n,k}\|^2|\F_n]
        \le \frac{2\gamma_1^2 L}{N} \|A z_n\|^2.
    \end{align*}
    Summing over $k$ from $k =0$ to $k = m-1$ gives
    \begin{align*}
        \EE[\|z_{n+1}\|^2|\F_n] - \|z_{n,0}\|^2
        \le \frac{2m\gamma_1^2 L}{N} \|A z_n\|^2.
    \end{align*}
    Combining this with (\ref{svrg.211}) gives 
    \begin{align*}
        \EE[\|z_{n+1}\|^2|\F_n] - \|z_n \|^2
        \le -\left(2\gamma_0 - \gamma_0^2 \|A\|^2 -\frac{2m\gamma_1^2 L}{N}\right) \|A z_n\|^2\le 0
    \end{align*}
    which, by taking the full expectation, implies $\EE[\|z_{n+1}\|^2] \le \EE[\|z_n\|^2]$ for all 
    integers $n \ge 0$. By recursively using this inequality we thus obtain $\EE[\|z_n\|^2] \le \|z_0\|^2 = \|x_0 - \hat x_0\|^2$. The proof is complete. 
\end{proof}
 
Based on Lemma \ref{SVRG:lem21} and Lemma \ref{SVRG:lem22}, we can now prove the convergence of 
Algorithm \ref{alg:SVRG} with exact data. 

\begin{theorem}\label{SVRG:thm6}
    Consider the sequence $\{x_n\}$ defined by Algorithm \ref{alg:SVRG} with exact data. Assume that 
    $\gamma_0>0$ and $\gamma_1>0$ are chosen such that (\ref{SVRG.5}) holds. Let $x^\dag$ denote the 
    unique $x_0$-minimal norm solution of (\ref{SVRG.1}). Then 
    $$
        \lim_{n\to \infty} \EE\left[\|x_n - x^\dag\|^2\right] = 0.
    $$
\end{theorem}

\begin{proof}
Since $x^\dag$ is the $x_0$-minimal norm solution of (\ref{SVRG.1}), there holds $x^\dag - x_0 \in \overline{\mbox{Ran}(A^*)}$. Thus for any $\ep>0$ we can find $\hat x_0 \in X$ 
such that $\|x_0-\hat x_0\| < \ep$ and $x^\dag - \hat x_0 \in \mbox{Ran}(A^*)$. Define $\{\hat x_n\}$
by Algorithm \ref{alg:SVRG} with exact data and with $x_0$ replaced by $\hat x_0$. Then from 
Lemma \ref{SVRG:lem22} it follows that 
$$
\EE\left[\|x_n - \hat x_n\|^2\right] \le \|x_0 - \hat x_0\|^2 < \ep^2. 
$$
Moreover, from Lemma \ref{SVRG:lem21} we have
$$
\EE\left[\|\hat x_n - x^\dag\|^2\right] \le C (n+1)^{-1}
$$
for some constant $C$ which may depend on $\ep$ but is independent of $n$. Consequently
\begin{align*}
\EE[\|x_n - x^\dag\|^2] 
& \le \EE\left[(\|x_n - \hat x_n\| + \|\hat x_n - x^\dag\|)^2\right] \\
& \le 2\EE\left[\|x_n - \hat x_n\|^2 + \|\hat x_n - x^\dag\|^2\right] \\
& \le 2 \ep^2 + 2C (n+1)^{-1}.  
\end{align*}
Therefore 
$$
\limsup_{n\to \infty} \EE\left[\|x_n - x^\dag\|^2\right] \le 2\ep^2. 
$$
Since $\ep>0$ is arbitrary, we must have $\EE\left[\|x_n - x^\dag\|^2\right] \to 0$ 
as $n \to \infty$. 
\end{proof}

By using Theorem \ref{SVRG:thm6} and Lemma \ref{SVRG:lem6} we are now ready to prove the main 
convergence result on Algorithm \ref{alg:SVRG} under an {\it a priori} stopping rule. 

\begin{theorem}\label{SVRG.thm4}
Consider Algorithm \ref{alg:SVRG}, where $\gamma_0>0$ and $\gamma_1>0$ are chosen such that 
(\ref{SVRG.5}) holds. Let $x^\dag$ denote the unique $x_0$-minimal norm solution of (\ref{SVRG.1}). 
Then for the integer $n_\d$ chosen such that $n_\d \to \infty$ and $\d^2 n_\d \to 0$ as $\d \to 0$ 
there holds 
$$
\EE[\|x_{n_\d}^\d - x^\dag\|^2] \to 0 \quad \mbox{ as } \d \to 0.
$$
\end{theorem}

\begin{proof}
We first have 
\begin{align*}
\EE\left[\|x_{n_\d}^\d - x^\dag\|^2\right] 
& \le 2 \EE\left[\|x_{n_\d}^\d - x_{n_\d}\|^2\right] 
+ 2 \EE\left[\|x_{n_\d} - x^\dag\|^2\right].
\end{align*}
Since $n_\d\to \infty$, we may use Theorem \ref{SVRG:thm6} to obtain 
$$
\EE\left[\|x_{n_\d} - x^\dag\|^2\right] \to 0 \quad \mbox{ as } \d \to 0. 
$$
By using Lemma \ref{SVRG:lem6} and $\d^2 n_\d \to 0$, we also have 
$$
\EE\left[\|x_{n_\d}^\d - x_{n_\d}\|^2\right]
\le C_0 \d^2 n_\d \to 0 \quad \mbox{ as }\d\to 0. 
$$
Therefore $\EE[\|x_{n_\d}^\d - x^\dag\|^2] \to 0$ as $\d \to 0$. 
\end{proof}

\section{\bf The discrepancy principle}\label{sect5}

The convergence results on Algorithm \ref{alg:SVRG} given in Theorem \ref{svrg.thm} and Theorem \ref{SVRG.thm4} are established under {\it a priori} stopping rules. In applications, we usually expect to terminate the iteration by {\it a posteriori} rules. Note that $r_n^\d:=A x_n^\d - y^\d$ is involved in the algorithm in every epoch, it is natural to consider terminating the iteration by the discrepancy principle which determines $n_\d$ to be the first integer such that $\|r_{n_\d}^\d\| \le \tau \d$, where $\tau>1$ is a given number. Incorporating the discrepancy principle into Algorithm \ref{alg:SVRG} leads to the following algorithm. 

\begin{algorithm}
        \caption{SVRG with the discrepancy principle}\label{alg:SVRG-DP}
	\begin{algorithmic}[0]	
		\STATE \textbf{input:} update frequency $m$, initial guess $x_0 \in X$, numbers $\tau>1$, $\gamma_0>0$, 
         $\gamma_1>0$. 
		
		\STATE {\bf for} $n = 0, 1, \cdots$ {\bf do}  
  
		\STATE \qquad Calculate $r_n^\d := A x_n^\d - y^\d$ \\[1ex]
  
        \STATE \qquad Set $\displaystyle{\mu_n := \left\{\begin{array}{lll}
        1  & \mbox{ if } \|r_n^\d\| > \tau \d\\
        0  & \mbox{ if } \|r_n^\d\| \le \tau \d;
        \end{array} \right.}$ \\[1ex]
        
		\STATE \qquad $g_n^\d = A^* r_n^\d$; \\[1ex]
  
        \STATE \qquad $x_{n,0}^\d  = x_n^\d  - \gamma_0 \mu_n g_n^\d$; \\[1ex]
		
		\STATE \qquad {\bf for } $k = 0, \cdots, m-1$ {\bf do} 
		
		\STATE \qquad \qquad pick $i_{n,k} \in \{1, \cdots, N\}$ randomly via uniform distribution; \\[1ex]
		
		\STATE \qquad \qquad $\displaystyle{g_{n,k}^\d = A_{i_{n,k}}^*A_{i_{n,k}} (x_{n,k}^\d - x_n^\d) + \frac{1}{N} g_n^\d}$;\\[1ex] 
		
		\STATE \qquad \qquad  $\displaystyle{x_{n,k+1}^\d = x_{n,k}^\d - \gamma_1 \mu_n g_{n,k}^\d}$; \\[1ex]
		
		\STATE \qquad {\bf end for} 
		
		\STATE \qquad $x_{n+1}^\d = x_{n,m}^\d$; 
		
		\STATE {\bf end} 
	\end{algorithmic}
\end{algorithm}

Algorithm \ref{alg:SVRG-DP} is formulated in the way that it incorporates the discrepancy principle 
to define an infinite sequence $\{x_n^\d\}$, which is convenient for analysis below. In numerical 
simulations, the iteration actually is terminated as long as $\|r_n^\d\| \le \tau \d$ because the 
iterates are no longer updated. It should be highlighted that the stopping index depends crucially 
on the sample path and thus is a random integer. Note also that the step sizes $\gamma_0 \mu_n$ and 
$\gamma_1 \mu_n$ in Algorithm \ref{alg:SVRG-DP} are random numbers; this sharply contrasts to 
Algorithm \ref{alg:SVRG} where the step size $\gamma_0$ and $\gamma_1$ are deterministic. The following result shows that the discrepancy principle can terminate the SVRG method in finite many steps almost surely. 

\begin{proposition}\label{prop:DP}
Consider Algorithm \ref{alg:SVRG-DP}. If $\tau>1$, $\gamma_0>0$ and $\gamma_1>0$ are chosen such that $0<\gamma_1<1/L$ and
$$
c_1 := 2\gamma_0 - \frac{2\gamma_0}{\tau} - \gamma_0^2 \|A\|^2 - \frac{2m \gamma_1^2L}{N} 
- \frac{m \gamma_1}{2 N (1- \gamma_1 L) \tau^2} >0,
$$
then 
\begin{align}\label{RBCD.DP1}
\EE[\|x_{n+1}^\d - x^\dag\|^2] \le \EE[\|x_n^\d - x^\dag\|^2] 
- c_1 \EE\left[\mu_n \|A x_n^\d - y^\d\|^2\right] 
\end{align}
for all integers $n\ge 0$. Moreover, Algorithm \ref{alg:SVRG-DP} 
must terminate in finite many steps almost surely. 
\end{proposition}

\begin{proof}
By following the proof of Lemma \ref{SVRG.lem1} with minor modifications we can obtain 
\begin{align*}
\EE[\|x_{n+1}^\d - x^\dag\|^2|\F_n] - \|x_{n,0}^\d - x^\dag\|^2  
& \le - \frac{2\mu_n\gamma_1 (1- \mu_n \gamma_1 L)}{N} \sum_{k=0}^{m-1} \EE[\|A x_{n,k}^d - y^\d\|^2|\F_n] \\
& \quad \, +\frac{2\mu_n \gamma_1}{N} \d \sum_{k=0}^{m-1} \EE[\|A x_{n,k}^\d - y^\d\||\F_n]  \\
& \quad \, + \frac{2m\mu_n \gamma_1^2 L}{N} \|A x_n^\d - y^\d\|^2 \\
& \le \frac{m\gamma_1 \mu_n \d^2}{2N(1-\mu_n \gamma_1 L)} + \frac{2m\mu_n \gamma_1^2 L}{N} \|A x_n^\d - y^\d\|^2 \\
& = \frac{m\gamma_1 \mu_n \d^2}{2N(1-\gamma_1 L)} + \frac{2m\mu_n \gamma_1^2 L}{N} \|A x_n^\d - y^\d\|^2
\end{align*}
and 
\begin{align*}
\|x_{n,0}^\d - x^\dag\|^2 - \|x_n^\d - x^\dag\|^2 
& \le - (2\gamma_0 - \gamma_0^2\|A\|^2) \mu_n \|A x_n^\d - y^\d\|^2 \\
& \quad \, + 2 \mu_0 \gamma_0 \d \|A x_n^\d - y^\d\|
\end{align*}
By the definition of $\mu_n$ we have $\mu_n \d \le \frac{\mu_n}{\tau} \|A x_n^\d - y^\d\|$. Therefore 
\begin{align*}
\EE[\|x_{n+1}^\d - x^\dag\|^2|\F_n] - \|x_n^\d - x^\dag\|^2
\le - c_1 \mu_n \|A x_n^\d - y^\d\|^2
\end{align*}
Taking the full expectation gives (\ref{RBCD.DP1}). 

Next we show that the method must terminate after finite many steps almost surely. To see this, consider the event
$$
{\mathcal E} := \left\{\|A x_n^\d - y^\d \|>\tau \d \mbox{ for all integers } n \ge 0\right\}
$$
It suffices to show $\P({\mathcal E}) = 0$. By virtue of (\ref{RBCD.DP1}) we have 
\begin{align*}
c_1 \EE\left[\mu_n \|A x_n^\d - y^\d\|^2\right] \le \EE[\|x_n^\d - x^\dag\|^2] - \EE[\|x_{n+1}^\d - x^\dag\|^2] 
\end{align*}
and hence for any integer $l \ge 0$ that 
\begin{align}\label{RBCD.DP2}
c_1 \sum_{n=0}^l \EE\left[\mu_n \|A x_n^\d - y^\d\|^2\right]
\le \EE[\|x_0^\d - x^\dag\|^2] = \|x_0 - x^\dag\|^2<\infty. 
\end{align}
Let $\chi_{\mathcal E}$ denote the characteristic function of ${\mathcal E}$, i.e. $\chi_{\mathcal E}(\omega) =1$
if $\omega \in {\mathcal E}$ and $0$ otherwise. Then 
\begin{align*}
\EE\left[\mu_n \|A x_n^\d - y^\d\|^2\right]
\ge \EE\left[\mu_n \|A x_n^\d - y^\d\|^2\chi_{\mathcal E}\right] 
\ge \tau^2 \d^2 \EE[\chi_{\mathcal E}]
= \tau^2 \d^2 \P({\mathcal E}).
\end{align*}
Combining this with (\ref{RBCD.DP2}) gives 
$$
c_1 \tau^2 \d^2 (l+1) \P({\mathcal E}) \le \|x_0 - x^\dag\|^2
$$
for all $l \ge 0$ and hence $\P({\mathcal E}) \le \|x_0-x^\dag\|^2/(c_1 \tau^2 \d^2 (l+1)) \to 0$ 
as $l \to \infty$. Thus $\P({\mathcal E}) =0$ and the proof is complete. 
\end{proof}

Proposition \ref{prop:DP} demonstrates that along any sample path from an event with probability one
there always exists a finite integer $n_\d$ such that 
$$
\| A x_{n_\d}^\d - y^\d \| \le \tau \d < \|A x_n^\d - y^\d\|, \quad 0\le n < n_\d,
$$ 
i.e. the discrepancy principle terminates the SVRG method almost surely, provided $\tau$, $\gamma_0$ and 
$\gamma_1$ are chosen properly. In Section \ref{sect6} we will provide various numerical results to test 
the performance of the discrepancy principle when it is used to terminate the SVRG method. 

\section{\bf Numerical simulations}\label{sect6}

In this section, we provide numerical simulations to test the performance of the SVRG method. All the computations are performed on the linear ill-posed system 
\begin{equation}
   A_{i}x:= \int_{a}^{b}K(s_{i},t)x(t)dt = y(s_{i}),  \quad i = 1,\cdots, N
   \label{Fredholm}
\end{equation} 
derived from the Fredholm integral equation of the first kind on $[c,d]$ by sampling at $s_{i} \in [c,d]$ with $i = 1,\cdots, N$, where the kernel $K(s,t)$ is continuous on $[c,d] \times [a,b]$ and $s_{i} = (i-0.5)(d-c)/N$ for $i = 1,\cdots, N$. We employ the three model problems, called \texttt{phillips}, \texttt{gravity} and \texttt{shaw}, which are described in \cite{H2007}. The first one is mildly ill-posed and the last two are severely ill-posed.  The brief information on these three model problems is given below. 

\begin{example}[\texttt{phillips}]\label{exmple1}
This test problem is obtained by discretizing the Fredholm integral equation 
$y(s) = \int_{-6}^{6}K(s,t)x(t)dt$, $s\in [-6, 6]$, where the kernel and the sought solution are given by 
$K(s,t) = \rho(s-t)$ and $x^\dag(t) = \rho(t)$ with 
\begin{equation*}
\rho(t) = \left\{\begin{matrix}
		    1+\cos(\frac{\pi t}{3}), & |t|<3,\\ 
		    0, & |t|\geq 3.
	        \end{matrix}\right.
\end{equation*}
   
\end{example}

\begin{example}[\texttt{gravity}]\label{exmple2}
This test problem follows from the discretization of a one-dimensional model problem in gravity surveying 
$y(s) = \int_{0}^{1}K(s,t)x(t)dt$,  $s\in [0, 1]$ which aims to recover a mass distribution $x(t)$ located 
at depth $d$ from the measured vertical component of the gravity field $y(s)$ at the surface. The kernel and 
the sought solution are 
   \begin{equation*}
       \begin{split}
           K(s,t) = d\left [ d^{2}+(s-t)^{2}\right ]^{-\frac{3}{2}}, \quad 
           x^\dag(t) = \sin(\pi t)+\frac{1}{2}\sin(2\pi t).
       \end{split}
   \end{equation*}
We use $d = 0.25$ in our computation. 
\end{example}

\begin{example}[\texttt{shaw}]\label{exmple3}
This one-dimensional image restoration model uses $y(s) = \int_{-\pi/2}^{\pi/2}K(s,t)x(t)dt$, $s\in [-\pi/2, \pi/2]$, 
where the kernel and the sought solution are 
   \begin{equation*}
       \begin{split}
           & K(s,t) = \left ( \cos(s)+\cos(t) \right )^{2}\left ( \frac{\sin(u)}{u} \right )^{2}, 
           \quad u=\pi\left ( \sin(s)+\sin(t) \right ),\\
           & x^\dag(t) = 2\exp\left ( - 6\left ( t-0.8 \right )^{2} \right )+\exp\left ( -2\left ( t+0.5 \right )^{2} \right ).
       \end{split}
   \end{equation*}
\end{example}

In the following we test the performance of Algorithm \ref{alg:SVRG} and Algorithm \ref{alg:SVRG-DP} by 
considering these three model examples. Instead of the exact data $y :=(y_1, \cdots, y_N)$ with $y_i := A_i x^\dag$
for each $i$, we use the noisy data $y^\d = (y_1^\d, \cdots, y_N^\d)$ generated by
\begin{equation}
    y^\d_{i} = y_{i} + \delta_{rel} |y_{i}|\epsilon_{i}, \quad i = 1,\cdots, N,
    \label{NoiseData1}
\end{equation}
where $\delta_{rel}$ is the relative noise level and $\epsilon_{i}$, $i = 1,\cdots,N$, are standard Gaussian noise. 
The integrals involved in the computation are approximated by the midpoint rule based on the partition of $[a,b]$ 
into $M:=N$ subintervals of equal length. All the simulations are performed on a Mac Air with Apple M1 processors, 
8GB DDR4 RAM, and a 512GB SSD using MATLAB R2022a.


In the computed examples, we utilize the noisy data $y^{\delta}$ with three different relative noise levels $\delta_{rel} = 10^{-1},10^{-2}$ and $10^{-3}$ and execute the SVRG method with the initial guess $x_{0} = 0$ together with the step-sizes given by (\ref{step-size}) in Remark~\ref{remark_stepsize}. 
In order to have fair judgement on the performance of the method, all statistical quantities presented below are computed from 100 runs. 

\begin{figure}[h]
\centering
\subfigure{
\includegraphics[width=4.35cm]{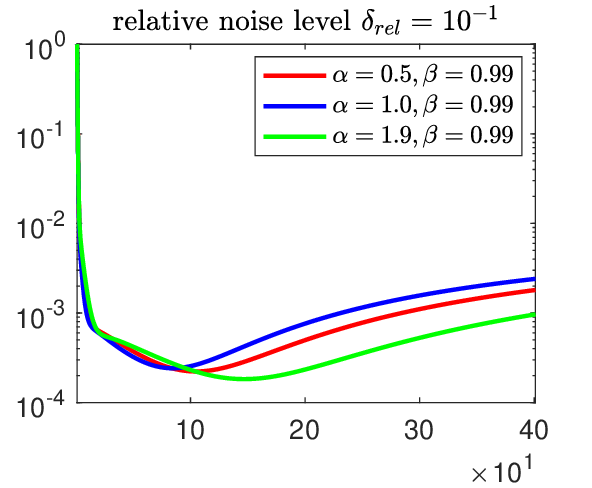}
}
\hspace{-7.8mm}
\vspace{-4mm}
\subfigure{
\includegraphics[width=4.35cm]{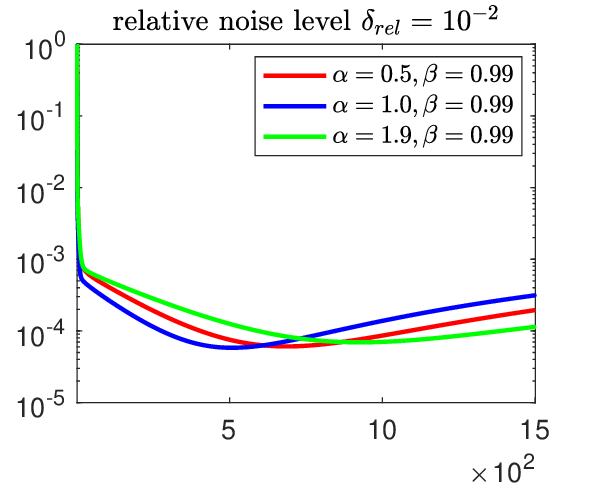}
}
\hspace{-7.8mm}
\subfigure{
\includegraphics[width=4.3cm]{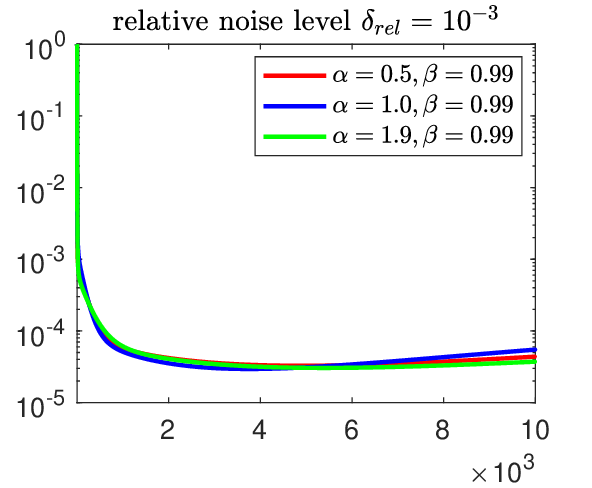}
}
\quad
\subfigure{
\includegraphics[width=4.35cm]{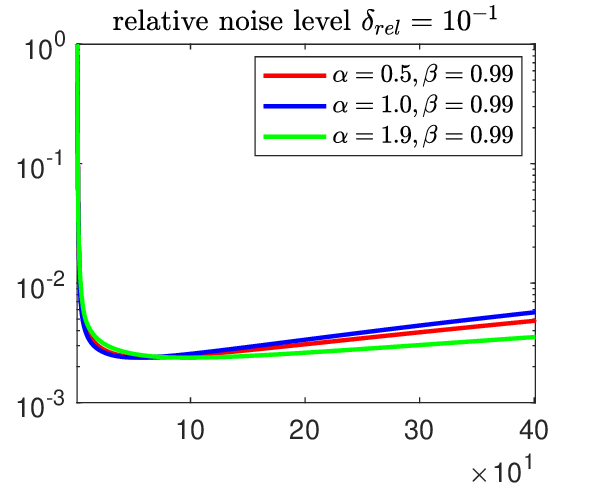}
}
\hspace{-7.8mm}
\vspace{-4mm}
\subfigure{
\includegraphics[width=4.35cm]{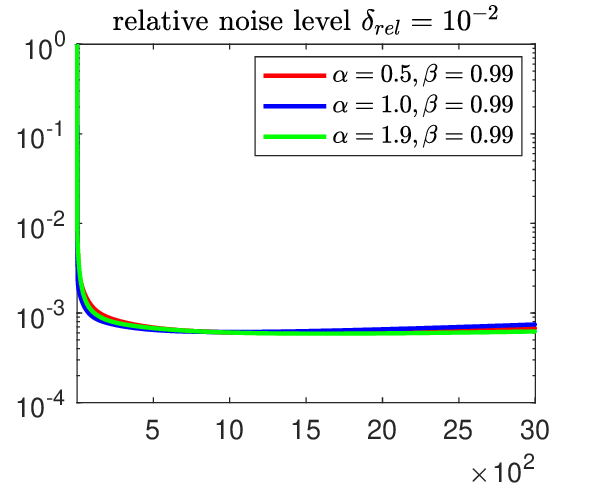}
}
\hspace{-7.8mm}
\subfigure{
\includegraphics[width=4.35cm]{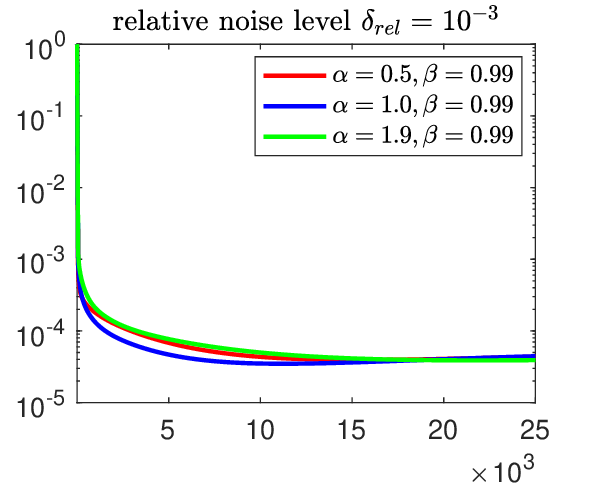}
}
\quad
\subfigure{
\includegraphics[width=4.35cm]{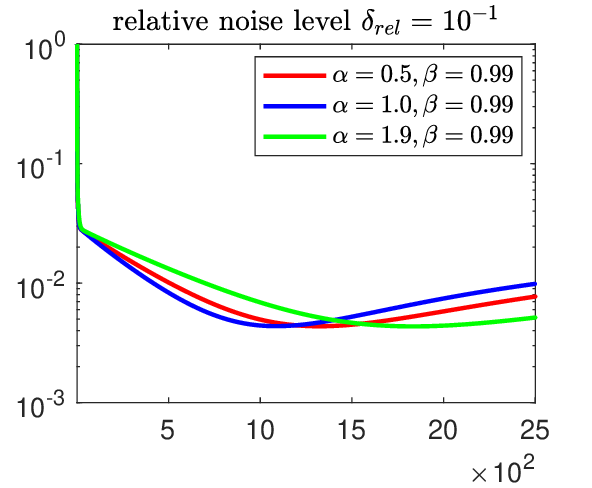}
}
\hspace{-7.8mm}
\subfigure{
\includegraphics[width=4.35cm]{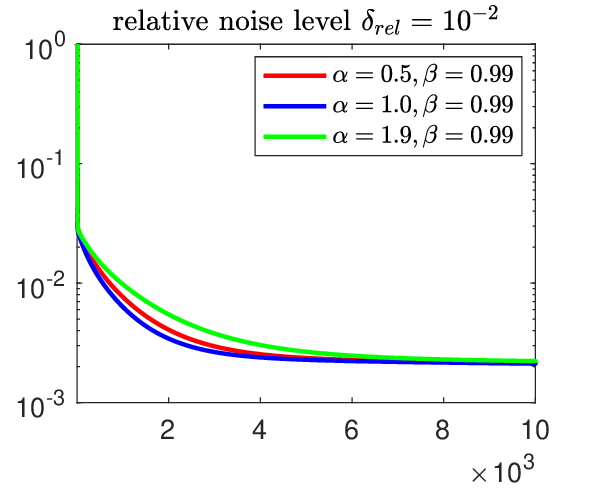}
}
\hspace{-7.8mm}
\subfigure{
\includegraphics[width=4.35cm]{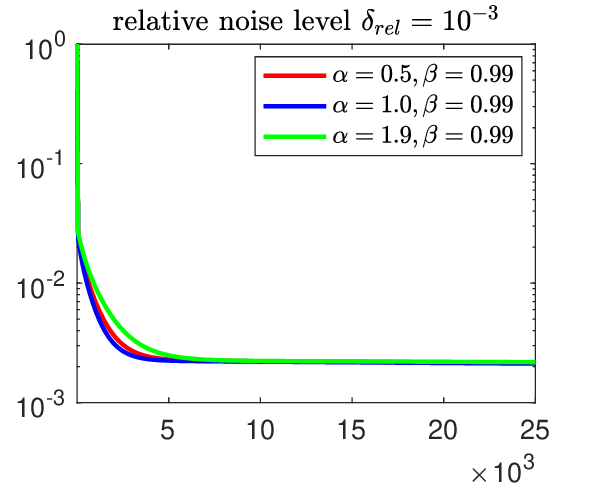}
}
\caption{Reconstruction error of SVRG using various parameters of $\alpha, \beta$ and the relative noise level $\delta_{rel}$. The rows from top to bottom refer to \texttt{phillips}, \texttt{gravity} and \texttt{shaw}, respectively.}
\label{ReMse}
\end{figure}

We first test the performance of Algorithm \ref{alg:SVRG} by considering the system \eqref{Fredholm} with $N=5000$. For a given discrete model, the step-sizes depend on the update frequency $m$, $\alpha$ and $\beta$. We use $m = 0.1N$. To illustrate the dependence of convergence on the magnitude of step-size, we consider the three groups of values: $(\alpha, \beta) = (0.5, 0.99), (1.0, 0.99)$ and $(1.9, 0.99)$. Figure \ref{ReMse} depicts the corresponding relative mean square errors $\mathbb{E}\left [ ||x_{n}^{\delta}-x^{\dagger}||^{2}/||x^{\dagger}||^{2}\right ]$ of reconstructions for the three model examples, where $n$ represent the number of epochs. These numerical plots demonstrate that the SVRG method exhibits the semi-convergence phenomenon, i.e., the iterate converges to the sought solution at the beginning and then starts to diverge after a critical number of iterations. Furthermore, the semi-convergence occurs earlier when the step-size is chosen by \eqref{step-size} with $(\alpha, \beta) = (1.0, 0.99)$ which means this choice of $(\alpha, \beta)$ allows the iterates to rapidly produce a reconstruction result with minimal error, but also quickly diverge from the sought solution. The semi-convergence behavior poses a challenge in determining how to terminate the iteration to produce satisfactory reconstruction results. It is therefore necessary 
to consider {\it a posteriori} stopping rules. 

\begin{table}[ht]
\caption{Numerical results for \texttt{phillips} model by SVRG, i.e. Algorithm \ref{alg:SVRG-DP} with 
$\gamma_0$ and $\gamma_1$ chosen by (\ref{step-size}) using $\a = 1$ and $\beta = 0.99$, 
and Landwber method (\ref{LM}) with $\gamma = 1/\|A\|^2$ terminated by the discrepancy principle with $\tau = 1.01$.}
\label{Phillips_DP}
\begin{center}
\begin{tabular}{llllll}
\hline
$N$ & $\d_{rel}$  & method     &  \texttt{iteration}  & \texttt{time} (s) &  \texttt{relative error}      \\
\hline
1000 & 0.1       & Landweber           &  19        &  0.0101    & 4.1590e-03  \\
~    & ~         & SVRG: $m = N$       &  2.72      &  0.0156    & 2.4393e-03 \\
~    & ~         & SVRG: $m = 0.1 N$   &  5.37      &  0.0080    & 3.4368e-03  \\
~    & 0.01      & Landweber           &  102       &  0.0385    & 7.9908e-04 \\
~    & ~         & SVRG: $m = N$       &  9.14      &  0.0525    & 1.1483e-03 \\
~    & ~         & SVRG: $m = 0.1 N$   &  22.21     &  0.0277    & 1.0987e-03  \\
~    & 0.001     & Landweber           &  3059      &  1.1237    & 9.6454e-05 \\
~    & ~         & SVRG: $m = N$       &  245.77    &  1.2181    & 1.1943e-04   \\
~    & ~         & SVRG: $m = 0.1 N$   &  638.62    &  0.6945    & 1.1686e-04  \\
 \hline
5000  & 0.1      & Landweber           &  16        &  0.1570    & 5.9102e-03  \\
~     & ~        & SVRG: $m = N$       &  2.03      &  0.4533    & 1.9841e-03  \\
~     & ~        & SVRG: $m = 0.1 N$   &  3.11      &  0.1041    & 3.9575e-03 \\
~     & 0.01     & Landweber           &  114       &  1.0965    & 6.5804e-04 \\
~     & ~        & SVRG: $m = N$       &  5.28      &  1.1646    & 9.2306e-04 \\
~     & ~        & SVRG: $m = 0.1 N$   &  13.41     &  0.4330    & 9.6879e-04  \\
~     & 0.001    & Landweber           &  2690      &  25.294    & 1.5558e-04  \\
~     & ~        & SVRG: $m = N$       &  93.57     &  20.208    & 1.6958e-04 \\
~     & ~        & SVRG: $m = 0.1 N$   &  283.17    &  9.1490    & 1.6979e-04  \\
\hline
10000 & 0.1      & Landweber           &  16        &  0.6462    & 6.3237e-03  \\
~     & ~        & SVRG: $m = N$       &  2.01      &  2.4825    & 1.7839e-03  \\
~     & ~        & SVRG: $m = 0.1 N$   &  2.7       &  0.4334    & 3.6413e-03  \\
~     & 0.01     & Landweber           &  116       &  3.9144    & 6.3945e-04  \\
~     & ~        & SVRG: $m = N$       &  3.85      &  4.6586    & 7.9873e-04  \\
~     & ~        & SVRG: $m = 0.1 N$   &  10.41     &  1.6302    & 8.7121e-04  \\
~     & 0.001    & Landweber           &  3449      &  116.48    & 9.0028e-05  \\
~     & ~        & SVRG: $m = N$       &  95.05     &  111.98    & 1.1096e-04  \\
~     & ~        & SVRG: $m = 0.1 N$   &  268.72    &  44.136    & 1.0459e-04  \\
\hline
\end{tabular}\\[5mm]
\end{center}
\end{table}

Next we assume that the information on the noise level $\d:=\|y^\d - y\|$ is available and consider the SVRG method terminated by the discrepancy principle as described in Algorithm \ref{alg:SVRG-DP}. We demonstrate the numerical performance of Algorithm~\ref{alg:SVRG-DP} on the three model problems with $\gamma_0$ and $\gamma_1$ chosen by (\ref{step-size}) with $\alpha = 1.0, \beta = 0.99$. In this study, we employ the Landweber method as our benchmark. For the comparison, the Landweber method (\ref{LM}) is initialized with $x_{0}=0$ with the constant step-size $\gamma = 1/||A||^{2}$. The both methods are terminated by the discrepancy principle with $\tau = 1.01$. 

\begin{table}[ht]
\caption{Numerical results for \texttt{gravity} model by SVRG, i.e. Algorithm \ref{alg:SVRG-DP} with 
$\gamma_0$ and $\gamma_1$ chosen by (\ref{step-size}) using $\a = 1$ and $\beta = 0.99$, 
and Landwber method (\ref{LM}) with $\gamma = 1/\|A\|^2$ terminated by the discrepancy principle with $\tau = 1.01$.}
\label{Gravity_DP}
\begin{center}
\begin{tabular}{llllll}
\hline
$N$ & $\d_{rel}$  & method     &  \texttt{iteration}  & \texttt{time} (s) &  \texttt{relative error}              \\
\hline
1000 & 0.1       & Landweber          &  23        &  0.0091    &   6.8214e-03  \\
~    & ~         & SVRG: $m = N$      &  2.52      &  0.0136    &   6.2835e-03  \\
~    & ~         & SVRG: $m = 0.1 N$  &  4.54      &  0.0055    &   7.5344e-03  \\
~    & 0.01      & Landweber          &  178       &  0.0618    &   2.0434e-03  \\
~    & ~         & SVRG: $m = N$      &  12.72     &  0.0625    &   2.0389e-03  \\
~    & ~         & SVRG: $m = 0.1 N$  &  34.03     &  0.0388    &   2.0621e-03  \\
~    & 0.001     & Landweber          &  3774      &  1.0686    &   3.1782e-04  \\
~    & ~         & SVRG: $m = N$      &  208.28    &  0.9407    &   3.1532e-04  \\
~    & ~         & SVRG: $m = 0.1 N$  &  649.56    &  0.7182    &   3.2604e-04  \\
\hline
5000 & 0.1       & Landweber          & 22         &  0.2075    &   7.7204e-03  \\
~    & ~         & SVRG: $m = N$      & 1.98       &  0.4385    &   5.6416e-03  \\
~    & ~         & SVRG: $m = 0.1 N$  & 2.89       &  0.0938    &   7.4545e-03  \\
~    & 0.01      & Landweber          & 249        &  2.2940    &   1.5475e-03  \\
~    & ~         & SVRG: $m = N$      & 8.71       &  1.9899    &   1.5585e-03  \\
~    & ~         & SVRG: $m = 0.1 N$  & 24.2       &  0.7469    &   1.6239e-03  \\
~    & 0.001     & Landweber          & 4588       &  42.369    &   2.7350e-04  \\
~    & ~         & SVRG: $m = N$      & 120.48     &  27.256    &   2.7895e-04  \\
~    & ~         & SVRG: $m = 0.1 N$  & 377.31     &  12.120    &   2.7368e-04  \\
\hline
10000 & 0.1      & Landweber          & 22         &  0.7390    &   8.0617e-03  \\
~    & ~         & SVRG: $m = N$      & 1.97       &  2.2943    &   4.6782e-03  \\
~    & ~         & SVRG: $m = 0.1 N$  & 2.38       &  0.3695    &   6.5503e-03  \\
~    & 0.01      & Landweber          & 275        &  9.0302    &   1.3574e-03  \\
~    & ~         & SVRG: $m = N$      & 6.53       &  7.6368    &   1.4864e-03  \\
~    & ~         & SVRG: $m = 0.1 N$  & 17.73      &  2.6020    &   1.4652e-03  \\
~    & 0.001     & Landweber          & 4614       &  153.75    &   2.7504e-04  \\
~    & ~         & SVRG: $m = N$      & 89.83      &  107.45    &   2.7817e-04  \\
~    & ~         & SVRG: $m = 0.1 N$  & 288.95     &  43.520    &   2.7620e-04  \\
\hline
\end{tabular}\\[5mm]
\end{center}
\end{table}

The SVRG algorithm involves a hyperparameter, the update frequency $m$ of evaluating the full gradient, which is a key parameter for the performance and efficiency of the method. To assess the impact of the hyperparameter $m$ at different scales, we conduct a series of numerical experiments with $m=N$ and $m=0.1N$ for three different discretization levels $N= 1000, 5000, 10000$. The numerical results for the three model problems are reported in Table \ref{Phillips_DP}, Table \ref{Gravity_DP} and Table \ref{Shaw_DP}. In these tables, ``\texttt{iteration}" represents the stopping index output by the discrepancy principle, and ``\texttt{time}" and ``\texttt{relative error}" report the corresponding execution time and the relative error at the output stopping index; for Algorithm \ref{alg:SVRG-DP} these quantities are calculated as the averages of 100 independent runs.  


\begin{table}[ht]
\caption{Numerical results for \texttt{shaw} model by SVRG, i.e. Algorithm \ref{alg:SVRG-DP} with 
$\gamma_0$ and $\gamma_1$ chosen by (\ref{step-size}) using $\a = 1$ and $\beta = 0.99$, 
and Landwber method (\ref{LM}) with $\gamma = 1/\|A\|^2$ terminated by the discrepancy principle with $\tau = 1.01$.}
\label{Shaw_DP}    
\begin{center}
\begin{tabular}{llllll}
\hline
$N$ & $\d_{rel}$  & method     &  \texttt{iteration}  & \texttt{time} (s) &  \texttt{relative error}              \\
\hline
1000 & 0.1       & Landweber          & 56      &   0.0183      & 3.3729e-02  \\
~    & ~         & SVRG: $m = N$      & 4.82    &   0.0242      & 3.2753e-02  \\
~    & ~         & SVRG: $m = 0.1 N$  & 11.94   &   0.0136      & 3.3493e-02  \\  
~    & 0.01      & Landweber          & 1732    &   0.5525      & 1.8242e-02  \\
~    & ~         & SVRG: $m = N$      & 137.64  &   0.6567      & 1.8157e-02  \\
~    & ~         & SVRG: $m = 0.1 N$  & 369.43  &   0.4109      & 1.8258e-02  \\
~    & 0.001     & Landweber          & 27018   &   6.9767      & 2.5595e-03  \\
~    & ~         & SVRG: $m = N$      & 2134.6  &   9.5064      & 2.5599e-03  \\
~    & ~         & SVRG: $m = 0.1 N$  & 5761.6  &   6.3493      & 2.5602e-03  \\
\hline
5000 & 0.1       & Landweber          & 57      &   0.5280      & 3.5610e-02  \\
~    & ~         & SVRG: $m = N$      & 2.75    &   0.6069      & 3.2465e-02  \\
~    & ~         & SVRG: $m = 0.1 N$  & 6.53    &   0.2047      & 3.4849e-02  \\
~    & 0.01      & Landweber          & 2743    &   25.450      & 1.4948e-02  \\
~    & ~         & SVRG: $m = N$      & 102.45  &   22.853      & 1.4832e-02  \\
~    & ~         & SVRG: $m = 0.1 N$  & 299.36  &   9.1962      & 1.4919e-02  \\
~    & 0.001     & Landweber          & 29136   &   283.34      & 2.4354e-03  \\
~    & ~         & SVRG: $m = N$      & 1077.3  &   240.26      & 2.4347e-03  \\
~    & ~         & SVRG: $m = 0.1 N$  & 3152.11 &   100.75      & 2.4353e-03  \\
\hline
10000 & 0.1      & Landweber          & 58      &   1.9163      & 3.4329e-02  \\
~    & ~         & SVRG: $m = N$      & 2.13    &   2.4791      & 3.0960e-02  \\
~    & ~         & SVRG: $m = 0.1 N$  & 5.02    &   0.7598      & 3.3105e-02  \\
~    & 0.01      & Landweber          & 3185    &   109.54      & 1.3195e-02  \\
~    & ~         & SVRG: $m = N$      & 84.47   &   100.94      & 1.3165e-02  \\
~    & ~         & SVRG: $m = 0.1 N$  & 252.86  &   41.515      & 1.3138e-02  \\
~    & 0.001     & Landweber          & 29962   &   1048.3      & 2.4488e-03  \\
~    & ~         & SVRG: $m = N$      & 792.24  &   924.37      & 2.4479e-03  \\
~    & ~         & SVRG: $m = 0.1 N$  & 2366.75 &   371.18      & 2.4488e-03  \\
\hline
\end{tabular}\\[5mm]
\end{center}
\end{table}

The numerical results reveal several noteworthy observations. First, the results demonstrate that Algorithm \ref{alg:SVRG-DP} can be terminated after finite number of iterations and produces acceptable approximate solutions. 
Meanwhile, the relative error consistently decreases steadily as the noise level $\delta$ decreases, exhibiting the convergence behavior of the proposed method.
In terms of accuracy (measured by the relative mean squared error), SVRG is competitive with the classical Landweber method for the three model problems. In most cases, the corresponding relative errors for the both methods are fairly close, and  occasionally the relative error of the SVRG method can be even smaller than Landweber method. 
These observations are valid for all the examples, despite their dramatic difference in degree of ill-posedness and solution smoothness. 

Note that each epoch in Algorithm \ref{alg:SVRG-DP} consists of a one-step of Landweber iteration which has complexity $O(N M)$ and an inner loop with $m$ iterations which has complexity $O(mM)$. Thus, the total complexity for each epoch of Algorithm \ref{alg:SVRG-DP} is $O((N+m) M)$. 
Consequently, if the algorithm is executed $n$ outer loops, the total computational complexity is $\mathcal{O}(n(N+m)M)$. Specifically, when $m=N$, one epoch in Algorithm \ref{alg:SVRG-DP} is equivalent to executing $2$ iterations of the Landweber method; and when $m=0.1N$, it is equivalent to executing $1.1$ times of Landweber steps. Therefore, from the perspective of computational complexity, the SVRG method is much more efficient than the Landweber method. For instance, in the gravity model with $\delta_{rel}=10^{-3}, N=10000, m =0.1N$,  the SVRG method executes $288.95 \times 1.1 \approx 318$ iterations of the Landweber method. This is only about $1/14.5$ of the 4614 iterations required by the Landweber method. However, since MATLAB optimizes matrix operations efficiently internally, directly using matrix operations is usually much more efficient than using loop iteration of matrix elements for calculation. Therefore, in the case of a small sample size ($N=1000$), although the theoretical computational complexity of the SVRG method is much lower than that of the Landweber method, the execution time of the algorithm is slightly higher. Even so, when handling large-scale problems ($N=5000, 10000$), the SVRG method exhibits significant performance advantages over the traditional Landweber method, as traditional methods may require more iterations to achieve the same convergence standards in such scenarios. 

Meanwhile, from the perspective of execution time results, we notice that the hyperparameter $m$ has a significant impact on the overall efficiency of the SVRG algorithm.  In particular, a larger $m$ value means that more number of inner iterations need to be performed in each loop, which directly leads to a significant increase in the computational cost required for each epoch, especially when processing large-scale problems. 
In addition, too large $m$ makes little variance reduction at the final stage of each inner iteration part because the iterates could be far away the snapshot point, which is not favorable to the overall performance of the algorithm. 
By analyzing the experimental results of the three different models under the same relative noise levels $\delta_{rel}$ and discretization levels, we found that, compared to $m = N$, setting $m=0.1N$ significantly reduces execution time, achieving two to three times faster efficiency while maintaining accuracy comparable to the traditional Landweber method. This emphasizes the effectiveness of appropriately decreasing $m$ to reduce computational costs and obtain satisfactory reconstruction results without affecting the performance of the algorithm.

To further illustrate the performance of individual samples, we present the boxplots of the relative errors and epochs with the discretization level $N= 10000$ for 100 simulations in 
Figure \ref{boxplots_10000}. On the box, the central mark is the median, and the bottom and top edges of the box indicate the 25th and 75 percentiles, respectively; the whiskers extend to the most extreme data points the algorithm considers to be not outliers, and the outliers are plotted individually using the ``+" symbol. It is visible that the proposed method exhibits convergence. Meanwhile, we observe that the relative error $||x_{n_{\delta}}^{\delta}-x^\dag||^{2}/\|x^\dag\|^2$ increases with the relative noise level $\delta_{rel}$, and its distribution also broadens. However, the required number of epochs to fulfill the posteriori stopping indices decreases dramatically, as the relative noise level $\delta_{rel}$ increases, concurring with the preceding observation.

\begin{figure}[htbp]
\centering
\subfigure{
\includegraphics[width=3.3cm]{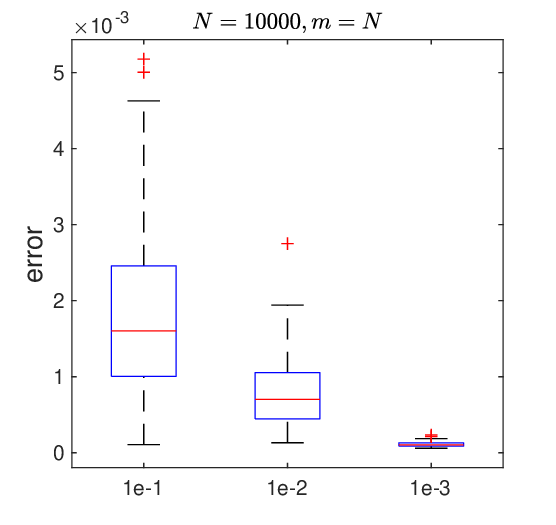}
}
\hspace{-7.2mm}
\vspace{-3mm}
\subfigure{
\includegraphics[width=3.3cm]{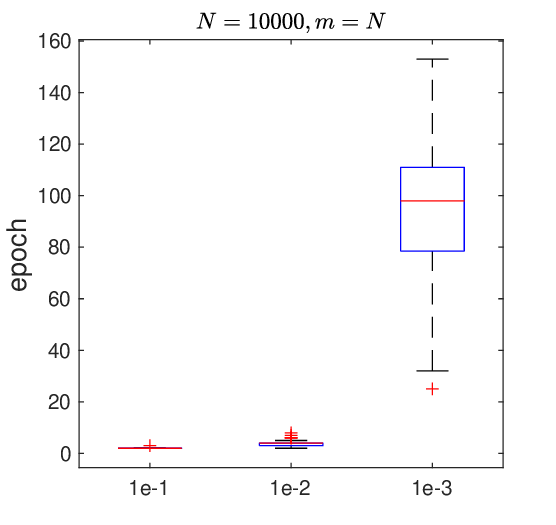}
}
\hspace{-7.2mm}
\subfigure{
\includegraphics[width=3.3cm]{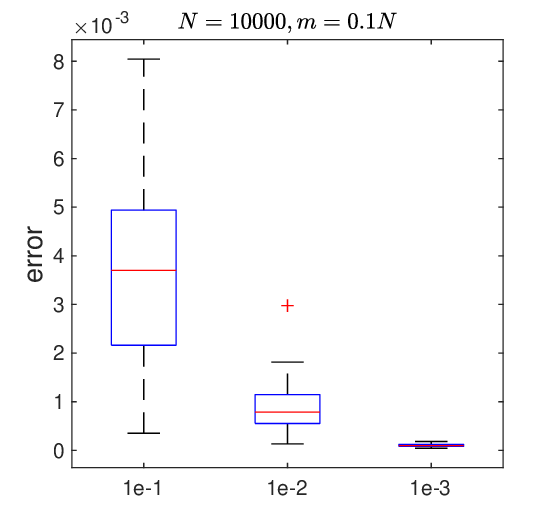}
}
\hspace{-7.2mm}
\subfigure{
\includegraphics[width=3.3cm]{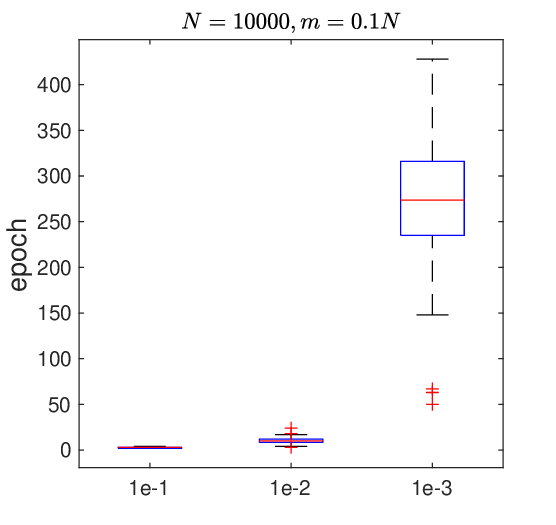}
}
\quad
\subfigure{
\includegraphics[width=3.2cm]{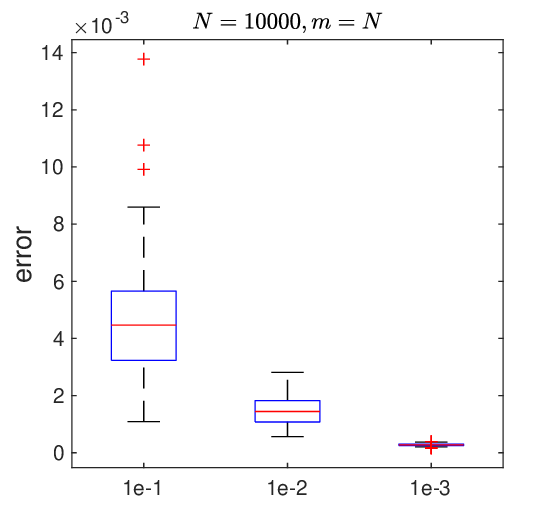}
}
\hspace{-7mm}
\vspace{-3mm}
\subfigure{
\includegraphics[width=3.2cm]{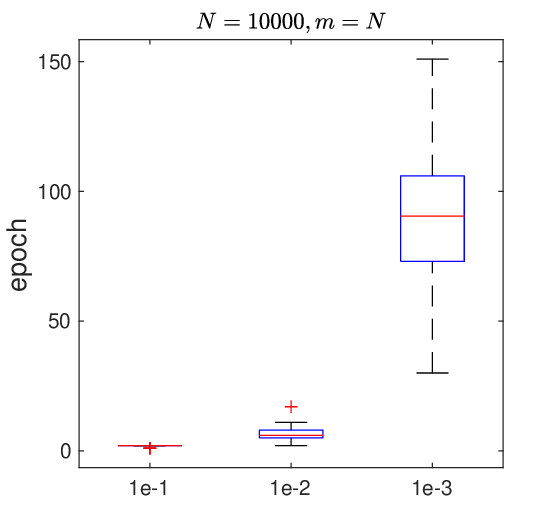}
}
\hspace{-7mm}
\subfigure{
\includegraphics[width=3.2cm]{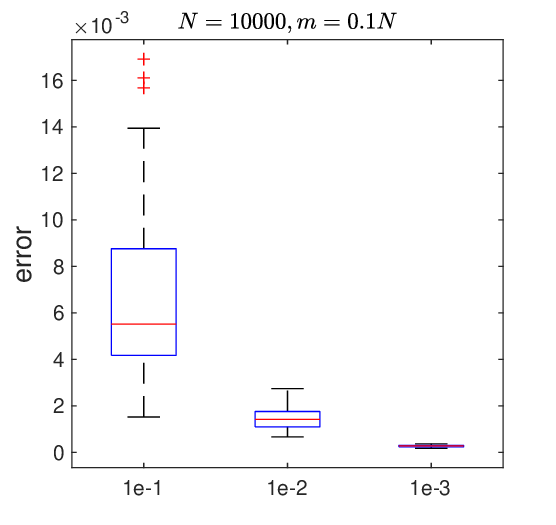}
}
\hspace{-7mm}
\subfigure{
\includegraphics[width=3.2cm]{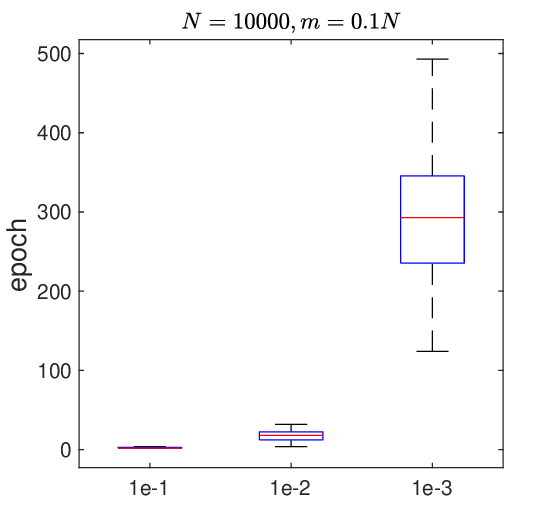}
}
\quad
\subfigure{
\includegraphics[width=3.2cm]{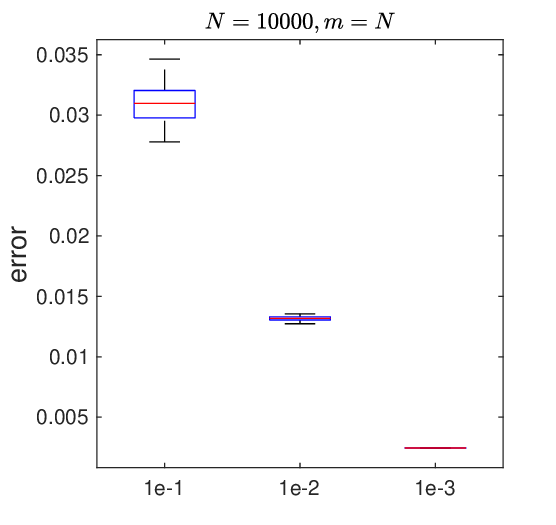}
}
\hspace{-7mm}
\subfigure{
\includegraphics[width=3.2cm]{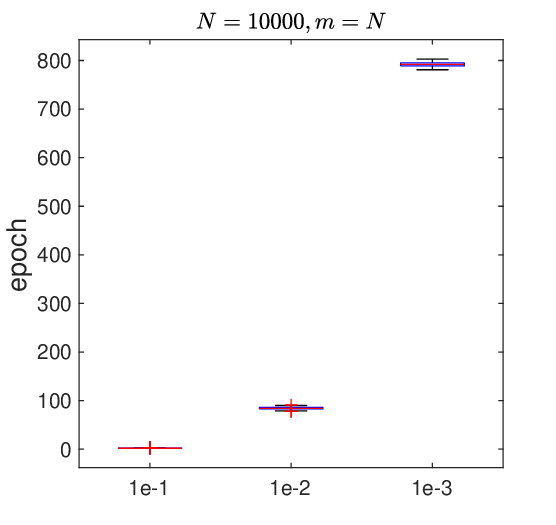}
}
\hspace{-7mm}
\subfigure{
\includegraphics[width=3.2cm]{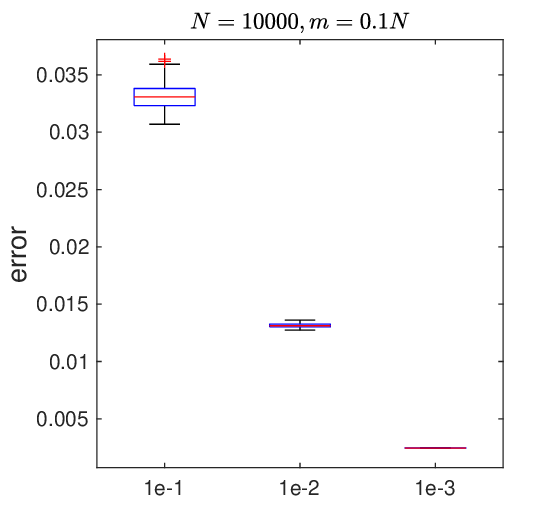}
}
\hspace{-7mm}
\subfigure{
\includegraphics[width=3.2cm]{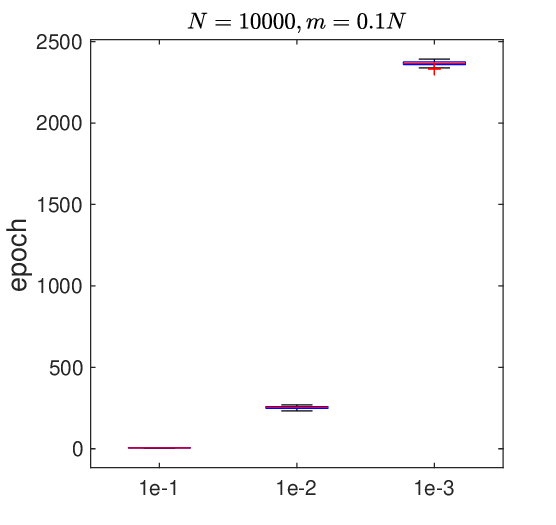}
}
\caption{Boxplots of the relative error $||x_{n_{\delta}}^{\delta}-x^{\dagger}||^{2}/||x^{\dagger}||^{2}$ 
and the stopping index $n_\d$ for the model probelms with $N=10000$. The rows from top to bottom refer 
to \texttt{phillips}, \texttt{gravity} and \texttt{shaw} respectively.}
\label{boxplots_10000}
\end{figure}
   
In order to visualize the performance, the reconstructed solutions of some individual runs are plotted in Figure \ref{reconstruction results}. All these results demonstrate that the proposed method consistently produces satisfactory reconstruction results.

\begin{figure}[ht]
\centering
  \includegraphics[width=0.35\textwidth]{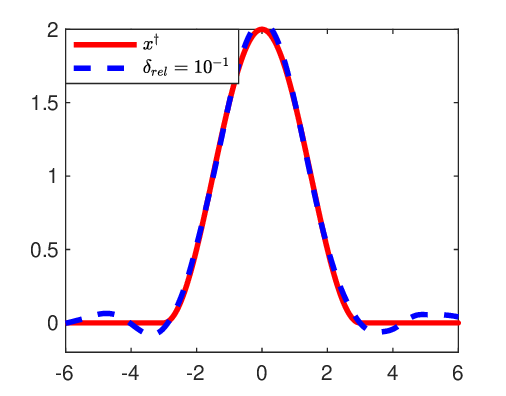}\hspace{-5.3mm}
  \includegraphics[width=0.35\textwidth]{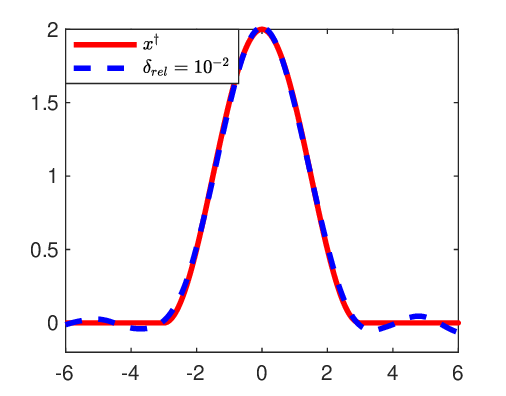}\hspace{-5.3mm}
  \includegraphics[width=0.35\textwidth]{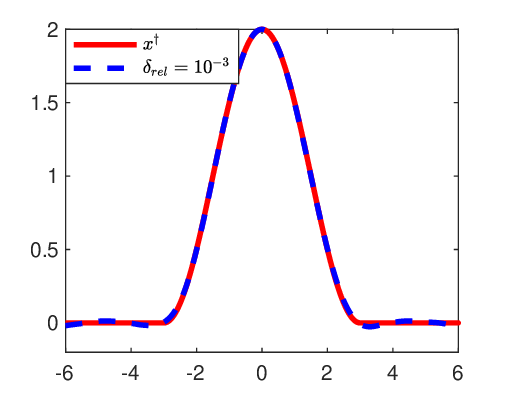}
  \includegraphics[width=0.35\textwidth]{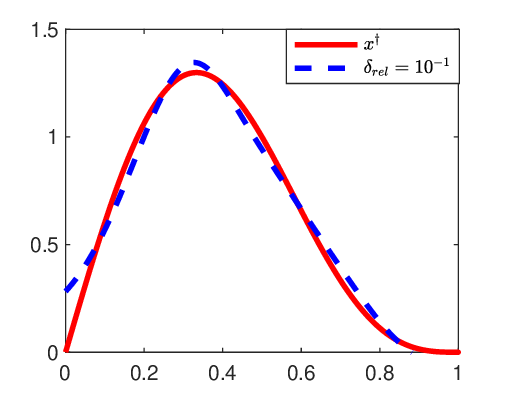}\hspace{-5.3mm}
  \includegraphics[width=0.35\textwidth]{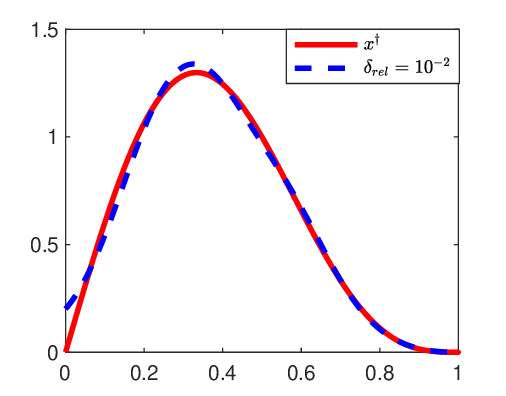}\hspace{-5.3mm}
  \includegraphics[width=0.35\textwidth]{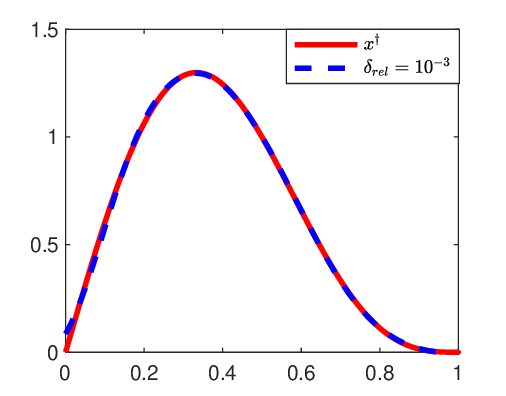}
  \includegraphics[width=0.35\textwidth]{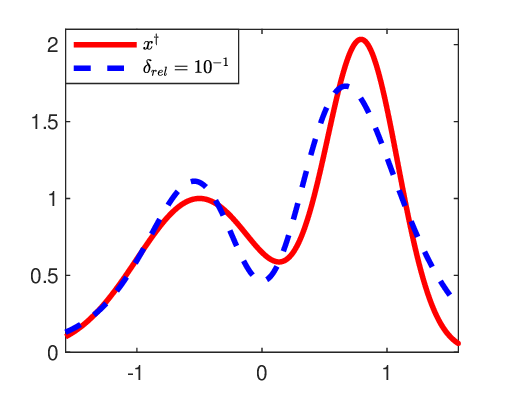}\hspace{-5.3mm}
  \includegraphics[width=0.35\textwidth]{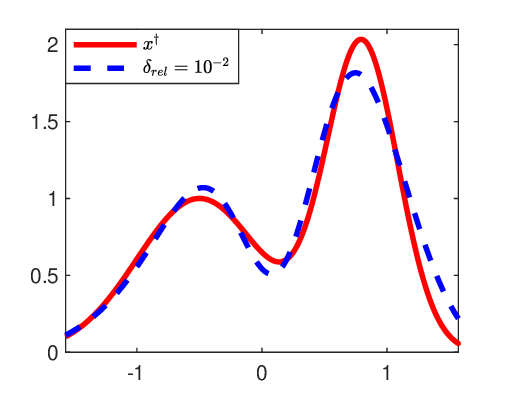}\hspace{-5.3mm}
  \includegraphics[width=0.35\textwidth]{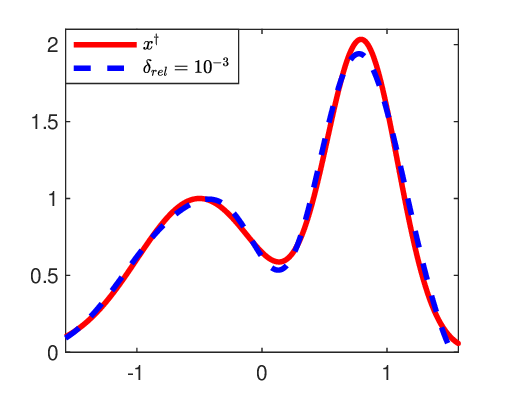}
\caption{The sought solution $x^{\dagger}$ and the reconstruction results by SVRG using noisy data 
with various relative noise levels. The rows from top to bottom refer to \texttt{phillips}, \texttt{gravity}
and \texttt{shaw} respectively.}
\label{reconstruction results}
\end{figure}

\section{\bf Conclusion}

Stochastic variance reduced gradient (SVRG) method is a prominent method for solving large scale well-posed 
optimization problems in machine learning and a variance reduction strategy has been introduced into the 
algorithm design to accelerate the stochastic gradient method. In this paper we applied the SVRG method to 
solve large scale linear ill-posed systems in Hilbert spaces. Under a benchmark source condition on the sought 
solution, we obtained a convergence rate result on the method when a stopping index is properly chosen. 
Based on this result and a perturbation argument we established a convergence result without using any 
source conditions. Furthermore, we considered the discrepancy principle to choose the stopping index and 
demonstrated that it terminates the SVRG method in finite many iteration steps almost surely. Various numerical 
results were reported which illustrate that the SVRG method can outperform the classical Landweber method
for large scale ill-posed inverse problems. There are several questions that might deserve further investigation: 

\begin{enumerate}[leftmargin = 0.8cm]
\item[$\bullet$]  In the SVRG method we used 
constant step sizes $\gamma_0$ and $\gamma_1$. Is it possible to develop a convergence theory of the 
SVRG method using adaptive step sizes so that larger step size can be allowed to reduce the number of 
iterations and hence to speed up the method? 

\item[$\bullet$] Our convergence theory for the SVRG method is for determining 
the $x_0$-minimal norm solutions. In applications, the sought solutions may have other {\it a priori}
available features, such as nonnegativity, sparsity and so on. Is it possible to modify the SVRG method 
with a solid theoretical foundation so that it can capture such desired features? 

\item[$\bullet$] Nonlinear ill-posed systems can arise from various tomography imaging problems. Can we 
extend the SVRG method to solve ill-posed systems in nonlinear setting? 
\end{enumerate}

\section*{\bf Acknowledgement}

The work of Q. Jin is partially supported by the Future Fellowship of the Australian Research Council (FT170100231).
The work of L. Chen is supported by the China Scholarship Council program (Project ID: 202106160067).


\begin{thebibliography}{10} 

\bibitem{AY2016} {\sc Z. Allen-Zhu and Y. Yuan}, {\it Improved SVRG for non-strongly-convex or 
sum-of-non-convex Objectives}, International Conference on Machine Learning, pp. 1080--1089, 2016.

\bibitem{BDP1985} M. Bertero, C.  De Mol and E. R. Pike,  {\it Linear inverse problems 
with discrete data. I: General formulation and singular system analysis}, Inverse Problems, 
l (1985),  301--330.

\bibitem{B2020} P. Br\'{e}maud, {\it Probability theory and stochastic processes}, Universitext, 
Springer, Cham, 2020. 

\bibitem{DBL2014} A. Defazio, F. Bach and S. Lacoste-Julien, {\it Saga: a fast incremental gradient 
method with support for non-strongly convex composite objectives}, Proc. 27th Int. Conf. Adv. Neural 
Inf. Process. Syst., pp. 1646--1654,  2014.

\bibitem{EHN1996} H. W. Engl, M. Hanke and A. Neubauer, {\it Regularization of Inverse Problems}, 
Dordrecht, Kluwer, 1996.

\bibitem{H2007} P. C. Hansen, {\it Regularization tools version 4.0 for Matlab 7.3},  Numer. Algorithms, 
46 (2007),  189--194. 

\bibitem{HAVSKS2015} R. Harikandeh, M. O. Ahmed, A. Virani, M. Schmidt,  J. Konecny and S. Sallinen, 
{\it Stop wasting my gradients: practical SVRG}, Proc. 28th Int. Conf. Adv. Neural Inf. Process. 
Syst.,  pp. 2251--2259, 2015.

\bibitem{JL2019}
B. Jin and X. Lu, {\it On the regularizing property of stochastic gradient descent}, Inverse Problems, 
35 (2019), no. 1, 015004, 27 pp.

\bibitem{JZZ2022} B. Jin, Z. Zhou and J. Zou, {\it An analysis of stochastic variance 
reduced gradient for linear inverse problems}, Inverse Problems, 38 (2022), no. 2, 025009, 34 pp. 

\bibitem{Jin2010} Q. Jin, {\it On a regularized Levenberg-Marquardt method for solving nonlinear inverse problems}, 
Numer. Math., 115 (2010), no. 2, 229--259. 

\bibitem{Jin2011} Q. Jin, {\it A general convergence analysis of some Newton-type methods for 
nonlinear inverse problems}, SIAM J. Numer. Anal., 49 (2011), no. 2, 549--573. 

\bibitem{JLZ2023} Q. Jin, X. Lu and L. Zhang, {\it Stochastic mirror descent method 
for linear ill-posed problems in Banach spaces}, Inverse Problems, 39 (2023), no. 6, 065010, 39 pp.

\bibitem{JZ2013} R. Johnson and T. Zhang, {\it Accelerating stochastic gradient descent using 
predictive variance reduction},  Advances in Neural Information Processing Systems, 315--323, 2013.

\bibitem{N2001} F. Natterer,{\it  The Mathematics of Computerized Tomography}, SIAM, Philadelphia, 2001.

\bibitem{NLST2017} L. M. Nguyen, J. Liu, K. Scheinberg and M. Tak\'{a}c, {\it SARAH: a novel method for 
machine learning problems using stochastic recursive gradient}, Proc. 34th Int. Conf. Machine Learning, 
PMLR, 70 (2017), 2613--2621. 

\bibitem{LSB2012} N. Le Roux, M. Schmidt and F. Bach, {\it A stochastic gradient method with an exponential 
convergence rate for strongly-convex optimization with finite training sets},  Advances in Neural 
Information Processing System, 2663--2671, 2012. 

\bibitem{SZ2013} S. Shalev-Shwartz and T. Zhang, {\it Stochastic dual coordinate as- cent methods for 
regularized loss minimization}, J. Mach. Learn. Res.,  14 (2013), 567--599. 

\bibitem{TMDQ2016} C. Tan, S. Ma, Y. H. Dai and Y. Qian, {\it Barzilai–Borwein step size for 
stochastic gradient descent},  Proc. 29th Int. Conf. Adv. Neural Inf. Process. Syst., 685--693,
2016. 

\bibitem{XZ2014} L. Xiao and T. Zhang, {\it A proximal stochastic gradient method with progressive variance
reduction}, SIAM Journal on Optimization, 24 (2014), no. 4, 2057–--2075.

\bibitem{YLDS2021} T. Yu, X. Liu, Y. H. Dai and J. Sun, {\it Stochastic variance reduced gradient methods 
using a trust-region-like scheme}, J. Sci. Comput., 87 (2021), no. 1, Paper No. 5, 24 pp. 

\bibitem{ZMJ2013} L. Zhang, M. Mahdavi and R. Jin, {\it Linear convergence with condition number 
independent access of full gradients}, Advances in Neural Information Processing Systems, 980--988, 2013. 

\bibitem{ZX2015} Y. Zhang and L. Xiao, {\it Stochastic primal-dual coordinate method for regularized 
empirical risk minimization},  Proc. Int. Conf. Mach. Learn., 2015, 353--361.
\end{thebibliography}
\end{document}